\documentclass{amsart}

\DeclareSymbolFont{bbold}{U}{bbold}{m}{n}
\DeclareSymbolFontAlphabet{\mathbbold}{bbold}

\usepackage{amsmath}%

\usepackage{amsfonts}%
\usepackage{amssymb}%
\usepackage{graphicx}
\usepackage{bm}

\usepackage{mathabx}
\usepackage{rotating}
\usepackage[all,cmtip]{xy}
\usepackage{cite}
\usepackage{hyperref}
\usepackage{geometry}
\usepackage{bbm}
\usepackage{stmaryrd}
\usepackage{tensor}
\usepackage{wasysym}
\usepackage{leftidx}
\usepackage{color}
\usepackage{comment}

\newtheorem{maintheorem}{Theorem}
\newtheorem{maincorollary}[maintheorem]{Corollary}

\newtheorem{dummy}{dummy}[section]              
\newtheorem{lemma}[dummy]{Lemma}
\newtheorem{theorem}[dummy]{Theorem}
\newtheorem{corollary}[dummy]{Corollary}
\newtheorem{proposition}[dummy]{Proposition}
\newtheorem{question}[dummy]{Question}

\theoremstyle{definition}                                  
\newtheorem{definition}[dummy]{Definition}
\newtheorem{example}[dummy]{Example}
\newtheorem{remark}[dummy]{Remark}

\newtheorem{convention}[dummy]{Convention}

\newtheorem*{example*}{Example}
\newtheorem*{remark*}{Remark}

\newcommand{\Vect}{\mathbf{Vect}}

\DeclareMathOperator{\Aut }{Aut}

\DeclareMathOperator{\bbHom}{Hom}

\DeclareMathOperator{\Fun}{Fun}
\DeclareMathOperator{\End}{End}

\DeclareMathOperator{\colim}{colim}

\newcommand{\im}{\mathrm{im}}

\newcommand{\module}{\mathrm{mod}}

\newcommand{\Perf}{\mathbf{Perf}}
\newcommand{\Coh}{\mathbf{Coh}}
\DeclareMathOperator{\Sym}{Sym}

\newcommand{\QC}{QC}

\DeclareMathOperator{\Ext}{Ext}
\DeclareMathOperator{\Rep}{Rep}
\newcommand{\Ad}{Ad}
\newcommand{\Cat}{\mathbf{Cat}}

\newcommand{\DGCat}{\mathbf{StPr}_{\C}^L}

\newcommand{\dgmod}{\mathbf{dgMod}}
\newcommand{\dgcomod}{\mathbf{dgComod}}

\newcommand{\too}{\longrightarrow}

\newcommand{\id}{\mathrm{id}}

\newcommand{\ls}[2]{\leftidx{^{#1}}{#2}{}}
\newcommand{\lrsub}[3]{\leftidx{_{#1}}{{#2}}{_{#3}}}
\newcommand{\lrsubsuper}[5]{\leftidx{^{#2} _{#1}}{{#3}}{_{#4} ^{#5}}}

\newcommand{\dw}{\dot{w}}

\newcommand{\bcD}{\breve{\bD}}

\DeclareMathOperator{\RES}{{\bf Res}}
\DeclareMathOperator{\IND}{{\bf Ind}}
\DeclareMathOperator{\Ind}{Ind}

\newcommand{\ind}{\mathbf{ind}}
\newcommand{\res}{\mathbf{res}}

\newcommand{\St}[2]{\lrsub{#1}{\underline{\mathfrak{st}}}{#2}}
\newcommand{\Stw}[3]{\leftidx{_{#1}}{\underline{\mathfrak{st}}}{^{#3} _{#2}}}

\DeclareMathOperator{\ST}{{\bf St}}

\newcommand{\Stein}{\mathfrak{s}\mathfrak{t}}

\newcommand{\Spr}{Spr}

\newcommand{\reg}{\mathrm{reg}}

\newcommand{\Levi}{\mathcal Levi}

\newcommand{\comm}{\mathfrak{comm}}

\newcommand{\rs}{\mathrm{rs}}

\newcommand{\fg}{\mathfrak{g}}

\newcommand{\ft}{\mathfrak{t}}
\newcommand{\fp}{\mathfrak{p}}
\newcommand{\fu}{\mathfrak{u}}
\newcommand{\fb}{\mathfrak{b}}

\newcommand{\fE}{\mathfrak{E}}
\newcommand{\fF}{\mathfrak{F}}
\newcommand{\fM}{\mathfrak{M}}
\newcommand{\fN}{\mathfrak{N}}
\newcommand{\fK}{\mathfrak{K}}

\newcommand{\fD}{\mathfrak{D}}

\newcommand{\fz}{\mathfrak{z}}
\newcommand{\fl}{\mathfrak{l}}
\newcommand{\fq}{\mathfrak{q}}

\newcommand{\ffi}{\mathfrak{i}}
\newcommand{\fj}{\mathfrak{j}}

\newcommand{\fgl}{\mathfrak{gl}}
\newcommand{\fsl}{\mathfrak{sl}}

\newcommand{\cC}{\mathcal C}
\newcommand{\cD}{\mathcal D}
\newcommand{\cE}{\mathcal E}
\newcommand{\cF}{\mathcal F}

\newcommand{\cH}{\mathcal H}

\newcommand{\cK}{\mathcal K}

\newcommand{\cN}{\mathcal N}
\newcommand{\cO}{\mathcal O}

\newcommand{\cQ}{\mathcal Q}

\newcommand{\cS}{\mathcal S}

\newcommand{\bH}{\mathbf{H}}
\newcommand{\bD}{\mathbf{D}}
\newcommand{\bM}{\mathbf{M}}

\newcommand{\ufD}{\underline{\fD}}

\newcommand{\A}{\mathbb A}

\newcommand{\C}{\mathbb C}
\newcommand{\D}{\mathbb D}

\newcommand{\G}{\mathbb G}

\newcommand{\R}{\mathbb R}

\newcommand{\Z}{\mathbb Z}

\newcommand{\ug}{{\underline{\mathfrak{g}}}}
\newcommand{\ul}{{\underline{\mathfrak{l}}}}
\newcommand{\up}{\underline{\mathfrak{p}}}

\newcommand{\uq}{\underline{\mathfrak{q}}}
\newcommand{\um}{{\underline{\mathfrak{m}}}}

\newcommand{\ub}{\underline{\mathfrak{b}}}
\newcommand{\uz}{\underline{\mathfrak{z}}}
\newcommand{\ut}{\underline{\mathfrak{t}}}
\newcommand{\ucN}{\underline{\mathcal{N}}}
\newcommand{\ucO}{\underline{\mathcal{O}}}

\newcommand{\bA}{\mathbf{A}}

\makeatletter
\newcommand*\leftdash{\rotatebox[origin=c]{-45}{$\dabar@\dabar@\dabar@$}}
\newcommand*\rightdash{\rotatebox[origin=c]{45}{$\dabar@\dabar@\dabar@$}}
\makeatother

\newcommand{\quot}[3]{{#1}\backslash{#2}/{#3}} 

\newcommand{\adjquot}{{/_{\hspace{-0.2em}ad}\hspace{0.1em}}}

\calclayout

\title[A Derived Decomposition]{A Derived Decomposition for equivariant $D$-modules}

\author{Sam Gunningham}

\begin{document}
	\begin{abstract}
		We show that the adjoint equivariant derived category of $D$-modules on a reductive Lie algebra $\fg$ carries an orthogonal decomposition in to blocks indexed by cuspidal data (in the sense of Lusztig). Each block admits a monadic description in terms a certain monad related to the homology of Steinberg varieties; this monad carries a filtration (the Mackey filtration) whose associated graded functor is given by the action of the relative Weyl group. Furthermore, we show that the Mackey filtration is generally non-split and thus the Springer-theoretic description of the entire equivariant derived category of $D$-modules appears to be substantially more subtle than either the case of the abelian category in earlier work of the author, or the derived category of nilpotent orbital sheaves in work of Rider and Russell. One notable feature of this setting is that the parabolic induction and restriction functors depend on the choice of parabolic subgroup containing a given Levi factor. 
	\end{abstract}

	\maketitle
	
	
	\section{Introduction}
\subsection{Main Results}

	Let $G$ be a connected, complex reductive group with Lie algebra $\fg$. We denote by $\ug = \fg/G$ the corresponding quotient stack, and $\bD(\ug)$ denotes the derived category of $D$-modules on $\ug$ (more precisely, its stable $\infty$-categorical enhancement). The goal of this paper is to understand the structure of this category.
	
	Recall that a \emph{cuspidal datum} $(L,\cE)$ for $G$, consists of a Levi subgroup $L$ of $G$ together with an irreducible cuspidal local system $\cE$ on a nilpotent orbit of $L$ in the sense of Lusztig \cite{lusztig_intersection_1984}. In \emph{loc. cit.}, Lusztig  has shown that cuspidal data index blocks of the abelian category $\bM(\ucN_G)$ of $G$-equivariant perverse sheaves on the nilpotent cone $\cN_G\subseteq \fg$. This was extended  by Rider and Russell to the equivariant derived category $\bD(\ucN_G)$ of constructible complexes on the nilpotent cone \cite{rider_perverse_2016}, and by the present author to the  abelian category $\bM(\ug)$ of equivariant $D$-modules on the entire Lie algebra \cite{Gunningham2018}. Our first result in this paper is a corresponding block decomposition for the \emph{derived} category $\bD(\ug)$.
	\begin{maintheorem}\label{maintheorem decomp}
		There is an orthogonal decomposition:
		\[
		\bD(\ug) \simeq \bigoplus _{(L,\cE)} \bD(\ug)_{(L,\cE)}
		\]
		indexed by conjugacy classes of cuspidal data $(L,\cE)$. 
	\end{maintheorem}
 	In the prior works mentioned above \cite{lusztig_intersection_1984, rider_perverse_2016,Gunningham2018},  the structure of the block corresponding to a given cuspidal datum $(L,\mathcal E)$ is determined by \emph{the relative Weyl} group $W_{G,L} := N_G(L)/L$, which is known to have the structure of a Coxeter group acting by reflections on the center $\fz(\fl)$ of the corresponding Levi subalgebra.
	
	Our next main result shows that the blocks of the derived category $\bD(\ug)$ are more subtle than that in the cases above. 
	\begin{maintheorem}\label{maintheorem monad}
		Let $(L,\cE)$ be a cuspidal datum for $G$, and choose a parabolic subgroup $P$ with Levi factor $L$. \begin{enumerate}
			\item \label{part monad} There is a monad $\ST = \ST_{(P,L,\cE)}$ acting on $\bD(\uz(\fl))$, and an equivalence
			\[
			\bD(\ug)_{(L,\cE)} \simeq \bD(\uz(\fl))^{\ST}.
			\]
			\item \label{part filtration} The functor $\ST$ carries a filtration (called the \emph{Mackey filtration}) indexed by $W_{G,L}$ (with Bruhat order induced by the choice of $P$), such that the associated graded functor is isomorphic to the group monad $(W_{G,L})_\ast$. 
		\end{enumerate}
	\end{maintheorem}
	
	As in \cite{Gunningham2018}, the proof of Theorem \ref{maintheorem decomp} and Theorem \ref{maintheorem monad} involves analyzing the functors of parabolic induction and restriction associated to a parabolic subgroup $P$ of $G$ with Levi factor $L$:
	\[
	\xymatrixcolsep{3pc}\xymatrix{
		\IND^G_{P,L}: \bD(\ul) \ar@/_0.3pc/[r] &  \ar@/_0.3pc/[l] \bD(\ug):  
		\RES_{P,L}^G
	}
	\]
	The composition of parabolic induction followed by restriction admits a filtration (the \emph{Mackey filtration}) in which the associated graded functors are given by further restrictions followed by inductions (see Proposition \ref{propositionmackey}). A crucial difference between our setting and that of \cite{lusztig_intersection_1984, rider_perverse_2016,Gunningham2018} is the fact that the Mackey filtration does not generally split. 
	
	To better understand this point, let us specialize to the case of the \emph{Springer block}, that is, the subcategory $\bD(\ug)_{\Spr} = \bD(\ug)_{(T,\C)}$ corresponding to the cuspidal datum $(T,\C)$ where $T$ is a maximal torus, and $\C$ the unique cuspidal local system on the nilpotent cone $\cN_T=pt$. In this case, after choosing a Borel subgroup $B$ containing $T$, Theorem \ref{maintheorem monad} states that 
	\[
	\bD(\ug)_{\Spr} \cong \bD(\ut)^{\ST}
	\] 
	Moreover, the Steinberg monad $\ST$ is given simply by pull-push along a correspondence
	\[
	\ut \leftarrow \St{}{} \rightarrow \ut,
	\]
	where $\St{}{}$ is the Steinberg stack (see Section \ref{sec:steinberg-stacks-and-functors}). The Mackey filtration on $\ST$ is induced by the natural stratification of $\Stein$ indexed by the Weyl group $W$ (or equivalently, the orbits $\quot BGB$). 

	\begin{maintheorem}\label{maintheorem nonsplit}
	Suppose $G$ is not abelian. Then the Mackey filtration on $\ST$ does not split. 
	\end{maintheorem}
	
	Theorem \ref{maintheorem nonsplit} has the following immediate consequence.
	\begin{maincorollary}
		\label{maincor nonisom} There does not exist an equivalence of categories $\varphi:\bD(\ug)_{\Spr}\simeq \bD(\ut)^W$ such that the following diagram commutes: 
		\[
		\xymatrix{
		\bD(\ug) \ar[r]^-\varphi \ar[rd]_-{\RES^G_{B,T}} & \bD(\ut)^W \ar[d]^-{\mathrm{Forget}}\\
		&\bD(\ut)
		}
		\]
	\end{maincorollary}

%
	
	The proof of Theorem \ref{maintheorem nonsplit} boils down to an explicit computation of the boundary map in a certain long exact sequence in Borel--Moore homology associated to the Hopf fibration (see Section \ref{sec:sl2 computation}). 
	
	
	Theorem \ref{maintheorem nonsplit} has the following interesting consequence concerning parabolic induction and restriction. For context, we recall that in many situations, the functors of parabolic induction and restriction are independent of the choice of parabolic subgroup containing a given Levi factor (see e.g. \cite{Gunningham2018}, Theorem B in the abelian category setting, and the work of Penghui Li in the Betti derived setting \cite{Li2018, Li2023}). On the other hand, in our setting, we have the following negative result.
	
	\begin{maintheorem}\label{mainthm opposite}
		Let $B$ and $\overline{B}$ be opposite Borel subgroups with respect to a maximal torus $T\subseteq G$, and suppose $G$ is non-abelian. Then there is no natural isomorphism of functors between $\IND^{G}_{B,T}$ and $\IND^G_{\overline{B},T}$ (or equivalently, between $\RES^G_{B,T}$ and $\RES^G_{\overline{B},T}$).
	\end{maintheorem}
	To deduce Theorem \ref{mainthm opposite} from Theorem \ref{maintheorem nonsplit} we use the second adjunction property of Braden, which states that $\RES^G_{\overline{B},T}$ is left adjoint to $\IND^G_{B,T}$ as well as being right adjoint to $\IND^G_{\overline{B},T}$ (see Theorem \ref{theorem second adjunction}). By Theorem \ref{maintheorem nonsplit}, if $\RES^G_{B,T}$ is naturally isomorphic to $\RES^G_{\overline{B},T}$, the counit of the resulting second adjunction cannot split the unit of the first adjunction. We derive a contradiction by analyzing behavior of these functors over the regular locus. 
	
	\begin{remark}
		The statement of Theorem \ref{mainthm opposite} is somewhat delicate, as the following observations indicate.
		\begin{enumerate}
			\item First let us note that the fact that $B$ and $\overline{B}$ are conjugate subgroups of $G$ does not contradict Theorem \ref{mainthm opposite}, as the conjugation will act non-trivially on the torus $T$. In the statement of Theorem \ref{mainthm opposite} it is crucial that we considered $T$ as a fixed subgroup of $G$ rather than as the universal torus $B/[B,B] \cong \overline{B}/[\overline{B},\overline{B}]$. Indeed, if we fix an element $g\in G$ such that $gBg^{-1} = \overline{B}$, then the induced isomorphism 
			\[
		\xymatrix{
		T \ar[r]_-\sim \ar@/^1pc/[rrr]& B/[B,B] \ar[r]^-{\Ad_g}_-\sim & \overline{B}/[\overline{B},\overline{B}] & \ar[l]^-\sim T
		}
			\]
			is realized by the action of the longest element $w_0 \in W$.
			
			\item It is also crucial that we take the \emph{equivariant} derived category $\bD(\ut) \simeq \bD(\ft)\otimes \bD(BT)$ as the source of the induction functor in Theorem \ref{mainthm opposite}. Indeed, it follows from work of Tsao-Hsien Chen \cite[Proposition 3.2]{Chen2021} that the induction functor $\IND':\bD(\ft) \to \bD(\ug)$ obtained by first forgetting the equivariance is, in fact, independent of the choice of Borel subgroup. The author would like to thank Roman Bezrukavnikov and Alexander Yom-Dim for some illuminating discussions surrounding these ideas.  
		\end{enumerate}	
	\end{remark}

	\subsection{Organization of the paper}
	\begin{itemize}
		\item In the remainder of the introduction below, we discuss some related results and related ideas.
		
		\item In Section \ref{sec:preliminaries} we give an overview of some of the tools we will need from the theory of stable $\infty$-categories and $D$-modules.
		
		\item In Section \ref{sec:mackey-theory-and-decomposition-by-levis} we review some of the ideas from \cite{Gunningham2018}, establish a derived recollement situation indexed by conjugacy classes of Levi subgroups, and show that it splits, leading to an orthogonal decomposition.
		
		\item In Section \ref{sec:generalized-springer-decomposition} we study the cuspidal blocks of the derived category, and complete the proof of Theorem \ref{maintheorem decomp}.
		
		\item In Section \ref{sec:the-steinberg-monad}, we study the blocks of the derived category via the Steinberg monad, and complete the proof of Theorem \ref{maintheorem monad}.
		
		\item In Section \ref{sec:nonsplit} we examine the splitting of Mackey filtration and complete the proof of Theorem \ref{maintheorem nonsplit}, Corollary \ref{maincor nonisom}, and Theorem \ref{mainthm opposite}. 
	\end{itemize}

	\subsection{Related work}\label{sec:background-and-motivation}
	A more extensive overview of the background, related works, and motivation for the study of $D$-modules on reductive groups and Lie algebras can be found in the author's other paper \cite{Gunningham2018}. In this subsection, we make a more concise comparison our results with a few closely related results in the literature.

	Lusztig's paper \cite{lusztig_intersection_1984} introduced the Generalized Springer Correspondence, which gave a block decomposition of the equivariant perverse sheaves on the nilpotent cone $\cN_G$, where the blocks are indexed by cuspidal data $(L,\cE)$ and each block is equivalent to the category of representations of the relative Weyl group $W_{G,L}$. The results of the present paper extend Lusztig's Generalized Springer Correspondence in two directions: 
	\begin{enumerate}
		\item replacing perverse sheaves on the nilpotent cone with (not-necessarily holonomic) $D$-modules on the whole of $\fg$;
		\item  replacing the abelian category with the derived category.
	\end{enumerate}
	The first direction is the subject of author's earlier paper \cite{Gunningham2018}, where it is shown that the abelian category $\bM(\ug)$ admits an orthogonal decomposition indexed by cuspidal data, and the blocks take the the following form:
	\[
	\bM(\fg)^G_{(L,\cE)} \simeq\bM(\fz(\fl))^{W_{G,L}}. 
	\] 
	
	The second direction has been studied by Achar \cite{achar_green_2011}, Rider \cite{rider_formality_2013}, and Rider--Russell \cite{rider_perverse_2016}. \footnote{In more recent work of Laumon and Letellier \cite{Laumon2023}, the authors consider the derived category of sheaves $\bD(\ug_{(L)})$ on a fixed stratum $\ug_{(L)}$ corresponding to a Levi subgroup $L$ of $G$ (see also \cite{Gunningham2018}, Definition 2.1 for the notation).} Note that, via the Riemann-Hilbert correspondence, we may identify the category $\bD(\ucN_G)$ of $G$-equivariant complexes of $D$-modules on $\fg$ with support in $\cN_G\subseteq \fg$ with the equivariant derived category of constructible complexes on $\cN_G$ (this identification uses the fact that $G$ acts on $\cN_G$ with finitely many orbits, and thus every $D$-module with support in $\cN_G$ is regular holonomic). In particular, it is shown in \cite{rider_perverse_2016} that the derived category $\bD_{coh}(\ucN_G)$ of equivariant constructible complexes on the nilpotent cone admits an orthogonal decomposition indexed by cuspidal data, and the blocks take the following form:
	\begin{equation}\label{eq:nilp gen spr}
	\bD(\ucN_G)_{(L,\cE)} 
	\simeq \bD_{coh}(pt/Z^\circ(L))^{W_{G,L}}.
	\end{equation}
	
	Theorem \ref{maintheorem decomp} extends these orthogonal decompositions to the entire category $\bD(\ug)$. One might naively guess that the blocks $\bD(\ug)_{(L,\cE)}$ are identified with $\bD(\uz(\fl))^{W_{G,L}}$. However Theorem \ref{maintheorem monad} states that this is not the case: rather, the naive guess is true up to taking the associated graded for a certain filtration.

	\begin{remark}
		In light of the present paper, the results of \cite{Gunningham2018} and \cite{rider_perverse_2016} can be understood as follows: the Mackey filtration \emph{does} split if we restrict to either the heart of the t-structure or the subcategory of complexes supported on the nilpotent cone. 
	\end{remark}

	\subsection{Rephrasing the results in terms of differential graded algebras}
		Let us rephrase Theorem \ref{maintheorem monad} in terms of differential graded-algebras (dg-algebras) as follows. We first set up some notation.

	Given a torus $Z$ with Lie algebra $\fz$, we write $\fD_{\fz}$ for the ring of differential operators and $S_{\fz} = \Sym(\fz^\ast [-2])$ for the symmetric algebra, understood as a (formal) differential graded algebra (dg-algebra) generated in degree 2. Let $\ufD_{\fz}$ denote the (formal) dg-algebra $\fD_{\fz} \boxtimes S_{\fz}$. Recall that the category $\bD_{coh}(\uz)$ of coherent complexes of $D$-modules on $\uz = \fz \times BZ$ is equivalent to the category $\ufD_{\fz}-\Perf$ of perfect dg $\ufD_{\fz}$-modules. Given a finite group $W$ acting on the stack $BZ$, we write $\ufD_{\fz} \mathbin{\#} W$ for the dg-vector space $\ufD_{\fz} \otimes \C[W]$, equipped with the smash product (that is, semidirect) dg-algebra structure.
	
	\begin{maintheorem}\label{maintheorem dg algebra}
		Let $(L,\cE)$ be a cuspidal datum for $G$ and choose a parabolic subgroup $P$ containing $L$ as a Levi factor. 
		\begin{enumerate}
			\item \label{part-bb} There is a dg-$\ufD_{\fz}$-ring $\bA = \bA_{(P,L,\cE)}$, together with an equivalence
			$
			\bD_{coh}(\ug)_{(L,\cE)} \simeq \bA-\Perf
			$.
			\item \label{part-mackey} There is a filtration (called the \emph{Mackey filtration}) of $\bA$ as $\ufD_{\fz(\fl)}$ dg-rings, indexed by $W_{G,L}$, such that the associated graded dg-algebra is 
			$
			\ufD_{\fz(\fl)}\mathbin{\#} W_{G,L}
			$
			$\bA \simeq \ufD_{\fz(\fg)}$.
		\end{enumerate}
	\end{maintheorem}

\subsection{Example: $SL_2$}\label{sec:example-sl2}
Most of the main features of our results can be seen in the case $G=SL_2$. In that case, there are two cuspidal data: the \emph{Springer datum} $Spr = (T,\C)$, where $T$ is a maximal torus of $G$, and $cusp = (G,\cE)$, where $\cE$ is the unique non-trivial local system on the regular nilpotent orbit of $G$. In this case Theorem \ref{maintheorem decomp} gives a decomposition of the form:
\[
\bD(\ug) \simeq \bD(\ug)_{Spr} \oplus \bD(\ug)_{cusp}
\]
The cuspidal block is generated by a single simple object $\fE = IC(\cO_{\reg};\cE) \in \bD(\ug)$, whose derived endomorphism algebra is concentrated in degree 0. Thus $\bD(\ug)_{cusp} \simeq \Vect$. This is in accordance with Theorem \ref{maintheorem monad}; indeed, here, $L=G$, $\uz(\fl)=pt$ and the Steinberg monad is trivial. 

On the other hand, by part \ref{part monad} of Theorem \ref{maintheorem monad}, the Springer block is identified with modules for the Steinberg monad $\ST$ acting on $\bD(\ut)$. After identifying endofunctors of $\bD(\ut)$ with integral kernels in $\bD(\ut\times \ut)$, the Steinberg monad is identified with the $D$-module of relative Borel-Moore chains, $f_\ast(\omega_{\Stein})$, where $f:\Stein \to \ut \times \ut$ is the \emph{Steinberg stack} (see Section \ref{sec:steinberg-stacks-and-functors}). The Steinberg stack carries an open-closed decomposition:
\[
\xymatrixcolsep{5pc}\xymatrix{
	 \Stein_e \ar[d]_-\wr \ar@{^{(}->}[r]^{\text{closed}} & \Stein \ar[d]_-f & \ar@{_{(}->}[l]_{\text{open}} \Stein_s \ar[d]^-\wr \\
	\ut \ar[r]_-{\Delta} & \ut \times \ut & \ar[l]^-{\nabla} \ut
	}
\]
The components are indexed by the Weyl group $W=\{e,s\}$, and $\Delta$ (respectively, $\nabla$) denotes the graph of the identity (respectively, the action of $s$) $\ut \to \ut \times \ut$. In this case, the Mackey filtration consists of a single exact triangle
\[
\Delta_\ast (\omega_{\ut}) \simeq f_\ast(\omega_{\Stein_e}) \to f_\ast(\omega_{\Stein}) \to f_\ast(\omega_{\Stein_s}) \simeq \nabla_\ast(\omega_{\ut})
\]
where $\Delta$ (respectively, $\nabla$) is the graph of the the identity (respectively $s$) $\ut \to \ut \times \ut$. In Section \ref{sec:sl2 computation} we show that this exact triangle is not split. This is the key calculation behind Theorem \ref{maintheorem nonsplit}. 

%

Equivalently, the dg-vector space $\bA$ appearing in Theorem \ref{maintheorem dg algebra} is given by taking global sections over $\ut \times \ut$ (in a certain sense; see Section \ref{sec: steinberg as integral}). Thus, there is an exact triangle
\begin{equation}\label{equation-triangle}
	\xymatrix{
		\bA_e \ar[r] & \bA \ar[r] & \bA_s \ar[r]^{+1} &  
	}
\end{equation}
where $\bA_e$ (respectively $\bA_s$) is the diagonal bimodule $\ufD_{\ft}$ (respectively the bimodule $\ufD_{\ft}$ where the left action is as usual and the right action is twisted by $s\in W$).

\subsection{A triple affine Hecke algebra?}\label{sec:a-triple-affine-hecke-algebra}
Theorem \ref{maintheorem dg algebra} gives a sense in which entire block $\bD_{coh}(\fg)^G_{(L,\cE)}$ has the flavor of a ``triple affine Hecke algebra'': two of the affine directions are in degree zero as represented by the copy of $\fD_{\fz(\fl)} = \Sym(\fz(\fl)) \rtimes \Sym(\fz(\fl)^\ast)$ sitting in $H^0(\bA_{(P,L\cE)})$, and the third is in even cohomological degrees, represented by copy of $S_{\fz(\fl)}$. 

The non-splitting of the Mackey filtration means that this algebra is not just a semidirect product of $W_{G,L}$ with $\ufD_{\fz(\fl)}$. This raises the following question, which the author hopes to return to in future work.
\begin{question}\label{questionhecke}
	Is there a combinatorial description of the dg-algebra $\bA_{(P,L,\cE)}$ in terms of the Coxeter system $(W_{G,L},\fz(\fl))$?
\end{question}
\begin{remark}
	It is not clear to the author if the dg-algebra is independent of the choice of parabolic $P$, or if it is formal (both of which seem to be necessary prerequisites to having any kind of reasonable combinatorial description).
\end{remark}

It is natural to look for deformations of the category $\bD_{coh}(\fg)^G$ which realize Hecke-type deformations of the algebras controlling the blocks. There are two flavors of such deformations. 

The first corresponds to deforming $H^0(\bA_{(P,L,\cE)}) = \fD_{\fz(\fl)} \rtimes W_{G,L}$ to a rational Cherednik algebra (``turning on the $c$ parameter''); the corresponding deformation of the abelian category $\bM(\fg)^G$ has only been understood geometrically in the case $G=GL_n$, in which case one studies the category of mirabolic $D$-modules $\bM(\fgl_n \times \C^n)^{GL_n,c}$ (moreover, only a generic block of that category, seen by Hamiltonian reduction, has been related to a Cherednik algebra). Even the abelian category story (as opposed to the derived categories studied in this paper) is a rich and active topic of research in type $A$ (see e.g. \cite{bellamy_hamiltonian_2015}) and there appears to be no analogue of the mirabolic deformation outside of type $A$. 

The second flavor of Hecke deformation corresponds to deforming $\bA_{(P,L,\cE),0} = S_{\fz(\fl)^\ast} \rtimes W_{G,L}$ to a graded Hecke algebra.We discuss this in Section \ref{sec:constructible-complexes-on-the-nilpotent-cone-and-graded-affine-hecke-algebras} below.

\subsection{Constructible complexes on the nilpotent cone, and graded affine Hecke algebras}\label{sec:constructible-complexes-on-the-nilpotent-cone-and-graded-affine-hecke-algebras}
	In the work of Rider \cite{rider_formality_2013}, formality of Springer block of the constructible derived category $\bD_{con}(\cN_G)^G$ had to be established first, before one could give a description of the category in terms of dg-modules. This was achieved by defining a mixed version of the category, which involves some intricate and technical constructions in the theory of triangulated categories. 
	
	The Barr--Beck--Lurie theorem allows us to construct equivalences as in Theorem \ref{maintheorem monad} without having any a priori formality results. 
	
	For example, it follows from the techniques in this paper that the Springer block of $\bD_{con}(\cN_G)^G$ is given by the dg-algebra of equivariant Borel-Moore chains on the Steinberg variety. Once this fact is established, it is not hard to prove that this dg-algebra is formal (e.g. by using Hodge theory), recovering Rider's identification of the Springer block. Note that these techniques do not address the construction of mixed enhancements of the category, which is of of significant independent interest.
	
	More generally, let us consider a block of the category $\bD_{con}(\cN_G)^G$ corresponding to a cuspidal datum $(L,\cE)$. It follows from the Barr-Beck-Lurie Theorem that this block is given by dg-modules for the dg-algebra
	$
	R\End(\IND^G_{P,L}\cE)
	$
	Once again this algebra is formal, and thus is given by the corresponding $\Ext$ algebra. As computed in \cite{rider_perverse_2016}, this algebra is given by the semidirect product
	$
	S_{\fz(\fl)} \rtimes W_{G,L}
	$.

	It is interesting to note that these semidirect product algebras which control the blocks of $\bD_{con}(\cN_G)^G$ deform to graded affine Hecke algebras once one considers $\G_m$-equivariant objects for the scaling $\G_m$-action on $\fg$. This is implicit in the work of Lusztig \cite{lusztig_cuspidal_1988, lusztig_cuspidal_1995}, where such algebras were first defined and identified with the $\Ext$ algebras of the parabolic induction of cuspidal local systems on Levi subgroups, or equivalently the twisted equivariant homology of certain Steinberg varieties. Putting Lusztig's results in the language of this paper, we obtain:
	\begin{theorem}
		 There is an orthogonal decomposition
		 \[
		 \bD_{coh}(\cN_G)^{G\times \G_m} = \bigoplus _{(L,\cE)} \bD_{coh}(\cN_G)^{G\times\G_m}_{(L,\cE)}
		 \]
		where each block is given by the category of dg-modules for the graded Hecke algebra $\bH_{(L,\cE)}$ with parameters as specified in \cite{lusztig_cuspidal_1988}.
	\end{theorem}
	\begin{remark}
		The motivation for considering these graded Hecke algebras comes from a relation with the representation theory of $p$-adic groups. The intersection cohomology complexes of local systems on nilpotent orbits for $G$ give examples of modules for the graded Hecke algebra.
	\end{remark}

	\subsection{Split recollement arising from the partition of the commuting variety}\label{sec:split-recollement-arising-from-the-partition-of-the-commuting-variety}
	As explained in \cite{Gunningham2018}, the decomposition of Theorem \ref{maintheorem decomp} (or rather the coarser decomposition indexed by the conjugacy class of the Levi $L$ as in Theorem \ref{thm: recollement is split}) has a geometric interpretation as follows. The variety  $\comm(\fg)$ of commuting elements of $\fg$ has a locally closed partition indexed by conjugacy classes of Levi subgroups. The singular support of an object in $\bD(\fg)^G$ is a closed subvariety of $\comm(\fg)$, and thus the category $\bD(\fg)^G$ carries a filtration indexed by conjugacy classes of Levis, according to the singular support. The somewhat surprising conclusion is that this filtration is, in fact, an orthogonal decomposition. 

\begin{remark}
Work of McGerty and Nevins \cite{mcgerty_morse_2014} identifies certain stratifications of the cotangent bundle of a stack which give rise to a recollement; although the results of this paper do not immediately apply in their setting, we expect that they are closely related (note that an orthogonal decomposition is a very special case of recollement).
\end{remark}

\subsection{Notation and conventions}\label{sec:notation}
The following overview of notational conventions may be helpful when reading this paper.
\begin{itemize}
	\item The following is a summary of the notation for $D$-modules (further details can be found in Section \ref{sec:ind-coherent-d-modules}). The abelian category of (all) $D$-modules on a smooth variety or stack $X$ is denoted $\bM(X)$; the (unbounded) dg derived category is denoted $\bD(X)$. The subcategory of bounded complexes with coherent cohomology modules is denoted $\bD_{coh}(X)$, and $\bcD(X) = \Ind \bD_{coh}(X)$ denotes the Ind-coherent category.
	\item Throughout the paper, $G$ always refers to a complex reductive group $P$ and $Q$ are parabolic subgroups with unipotent radicals $U$ and $V$ respectively, and $L = P/U$, $M=Q/V$ are the Levi factors. Lie algebras are denoted by fraktur letter $\fg, \fp , \fq$, etc. as usual. ***Come back to this... maybe want Levi subgroup to refer to a subgroup.
	\item The notation $Z^\circ(G)$ refers to the neutral component of the center of $G$, and $\fz(\fg)$ the center of the Lie algebra $\fg$ (which is identified with the Lie algebra of $Z^\circ(G)$).
	\item We use the underline notation $\ug$ to refer to the quotient stack $\fg/G$. Thus $\bD(\ug)$ means the same thing as $\bD(\fg/G)$ or equivalently $\bD(\fg)^G$. This notation may be applied in similar situations, e.g. $\up = \fp/P$, $\ucN_G = \cN_G/G$, $\uz(\fl)/Z^\circ(L)$ etc. Further notation for related stacks (e.g. the Steinberg stack $\St{Q}{P}$) is introduced in Section \ref{sec:steinberg-stacks-and-functors}.
	\item Functors of parabolic induction and restriction (introduced in Section \ref{sec:parabolic-induction-and-restriction}) will be denoted by $\IND^G_{P,L}$ and $\RES^G_{P,L}$. The Steinberg functors are the composites of parabolic induction with restriction, and are denoted by $\ST$.
	\item We frequently work with the poset $\Levi_G$ of Levi subgroups of $G$ up to conjugacy, ordered by inclusion. Thus $(M)\leq (L)$ means that some conjugate of $M$ is contained in $L$.   
	\item We use a superscript on the left to denote the adjoint action or conjugation. Thus $\ls gx$ means $\Ad(g)(x)$ (for some $g\in G$, $x\in \fg$), and if $P$ is a subgroup of $G$, then $\ls gP$ means $gPg^{-1}$.
\end{itemize}	

\begin{remark}
	References to numbered statements in \cite{Gunningham2018} correspond to those in arXiv:1510.02452v4, which has been rewritten to incorporate the published correction in 2021. (The main theorems A, B, and C are unchanged by the correction and are the same in the original published version and arXiv version 4). 
\end{remark}

\subsection{Acknowledgments}
Much of this paper was written while I was a postdoc at MSRI for the program on Geometric Representation Theory in fall 2014 and I would like to thank them for their hospitality. This work (which arose from my PhD thesis) has benefited greatly from numerous conversations with various people over the last few years. The following is a brief and incomplete list of such individuals whom I would especially like to thank, with apologies to those who are omitted. P. Achar, G. Bellamy, D. Ben-Zvi, D. Fratila, D. Gaitsgory,  D. Jordan, D. Juteau,  P. Li, C. Mautner, D. Nadler, L. Rider, T. Schedler. I would also like to thank two anonymous referees for their helpful comments on an earlier version. I would also like to thank T. Gannon, R. Bezrukavnikov, and A. Yom--Din for helpful discussions which prompted the later revisions and addition of Section 6 of this paper. The author was partially supported by NSF grant DMS-2202363.

\section{Preliminaries}\label{sec:preliminaries}
	In this section we set up the category theoretic framework, and describe the category of Ind-coherent $D$-modules on a stack.
	
	\subsection{Stable $\infty$-categories}\label{subsectionstable}
	In this paper we will make use of the theory of $\C$-linear, stable, presentable, $\infty$-categories, as developed by Lurie \cite{lurie_stable_2006} \cite{lurie_higher_2011}, or alternatively, pretriangulated differential graded categories (see \cite{cohn_differential_2013} for the relationship between the two theories). We refer the reader to \cite{ben-zvi_integral_2010} or \cite{gaitsgory_generalities_2012} for an overview of the main results and techniques, and directions towards further references. Below we outline some key properties.
	
	To each $\C$-linear, stable, presentable $\infty$-category $\cC$, the homotopy category $h\cC$ is a triangulated category; one thinks of $\cC$ as an enhancement of the triangulated category $h\cC$. For much of this paper, the reader may replace stable $\infty$-categories with their homotopy categories without any loss of understanding. However, the extra structure of these enhancements allow for much cleaner and more natural proofs of many of the results in this paper.
	
	\begin{definition}
		Let $\cC, \cD$ be stable, presentable, $\C$-linear $\infty$-categories. Given objects $c, d \in \cC$, we have a complex $R\bbHom(c,d)$ of morphisms from $c$ to $d$.
		\begin{enumerate}
			\item A functor $F: \cC \to \cD$ is called \emph{continuous} if it preserves all small colimits. 
			\item An object $c\in \cC$ is called \emph{compact} if the functor $R\bbHom(c,-)$ is continuous.
			\item We say $F$ is \emph{quasi-proper} if it sends compact objects to compact objects.
		\end{enumerate}
	\end{definition}

		We note that a functor $F:\cC\to \cD$ between stable $\infty$-categories is continuous if and only if it is exact (that is, it preserves all finite colimits, equivalently finite limits) and coproducts. In particular, an object $c$ in $\cC$ is compact if and only if $R\bbHom_\cC(c,-)$ preserves coproducts.

	The collection of stable presentable $\infty$-categories forms an $(\infty,1)$-category $\DGCat$, in which the morphisms are continuous functors. We also have $\infty$-categories $\Fun(\cC, \cD)$ and $\Fun^{L}(\cC,\cD)$, of functors and continuous functors respectively; both of these categories are themselves stable, presentable, and $\C$-linear.
	
	A category $\cC$ is called compactly generated if there is a subset $\cS$ of compact objects in $\cC$ such that the right orthogonal to $\cS$ vanishes. Given $\cC \in \DGCat$, we write $ \cC_c$ for the subcategory of compact objects. We can recover $\cC$ from $\cC_c$ as the Ind-category: $\cC \simeq \Ind (\cC_c)$. All the categories arising in this paper will be compactly generated.
	
	\begin{example}
		Given a differential graded (dg) algebra $A$, we have the stable $\infty$-category of perfect complexes $A-\Perf$, and $A-\dgmod = \Ind(A-\Perf))$ is the category of unbounded complexes of $A$-modules. In the special case $A= \C$, we write $\Vect := \C-\dgmod$.
	\end{example}
	
	The category $\DGCat$ carries a monoidal product $\otimes$, which is characterized by the property that continuous functors from $\cC \otimes \cD$ to $\Vect$ are the same thing as functors from $\cC \times \cD$ to $\Vect$ which are continuous in each argument separately. Given dg-algebras $A$ and $B$, we have 
	\[
	A-\dgmod \otimes B-\dgmod = A \otimes B-\dgmod.
	\]
	
	\begin{proposition}
		Suppose $\cC \in \DGCat$ is compactly generated category. Then $\cC$ is dualizable, with dual $\cC' := \Ind(\cC_c^{op})$.
	\end{proposition}
	
	Note that if $\cC$ is compactly generated then we have an equivalence
	\[
	\Fun^{L}(\cC, \cD) = \cC' \otimes \cD.
	\]
	
	Suppose $\cC, \cD \in \DGCat$ are compactly generated, and
	\[
	L: \cC \rightleftarrows \cD:R
	\]
	are an adjoint pair of functors (i.e. $L$ is left adjoint to $R$). Then $R$ is continuous (i.e. $R$ preserves small colimits) if and only if $L$ is proper (i.e. $L$ sends compact objects to compact objects). In that case, there is an adjunction
	\[
	L_c: \cC_c \rightleftarrows \cD_c: R_c.
	\]
	Conversely, if 
	\[
	L_c : \cC_c \rightleftarrows \cD_c : R_c
	\]
	is an adjoint pair of functors between the subcategories of compact objects, then we have an adjunction:
	\[
	\Ind(L_c): \cC \rightleftarrows : \cD : \Ind(R_c).
	\]
	
	\begin{definition}
		We say that a diagram 
		\[
		L: \cC \rightleftarrows \cD:R
		\]
		in $\DGCat$ is a \emph{continuous adjunction} if $L$ is left adjoint to $R$, and $R$ is continuous (equivalently, $L$ is quasi-proper).
	\end{definition}
	
	The following concept will be useful for us later.
	\begin{definition}\label{definitionfiltration}
		A \emph{filtration} of an object $a$ in a stable category $\cC$, indexed by a poset $(I,\leq)$, is a functor 
		\begin{align*}
		(I,\leq)              & \to \cC/a   \\
		i   & \mapsto a_{\leq i},
		\end{align*}
		where $\cC/a$ denotes the \emph{overcategory} (or \emph{slice category}) whose objects are pairs $(c,\pi)$ where $c$ is an object of $\cC$ and $\pi:c\to a$ a morphism in $\cC$, and morphisms $(a,\pi) \to (a'\pi')$ are given by morphisms $f:a\to a'$ in $\cC$ such that $\pi'f = \pi$. 
		
		In the cases of interest to us, $I$ will be a finite poset with a maximal element $i_{max}$, and we demand in addition that $a_{\leq i_{max}}\to a$ is an isomorphism. Given a closed subset $Z$ of $I$ (i.e. if $j\in Z$ and $i\leq j$, then $i\in Z$), we define $a_Z$ to be $\colim_{j\in Z} a_\leq j$. For any subset $J$ of $I$, we define $a_J$ to be the cone of $a_{< J} \to a_{\leq J}$. In particular, for any $i\in I$, we set $a_i$ to be the cone of $a_{<i} \to a_{\leq i}$. Thus, we think of the object $a$ as being built from $a_i$ by a sequence of cones. The \emph{associated graded} object is defined to be $\bigoplus_{i\in I} a_i$.
				
		Note that there is no requirement for the maps to be injective (indeed, it is not clear what that would mean in this setting).
	\end{definition}
%

	\subsection{The Barr--Beck--Lurie Theorem}\label{subsectionbbl}
	
	Recall that a \emph{monad} acting on a category $\cC \in \DGCat$ is an algebra object $T$ in the monoidal category $\Fun(\cC,\cC)$. Associated to this data we have the category of \emph{$T$-modules}\footnote{Some authors use the term $T$-algebra instead of $T$-module; we prefer the latter terminology as it is more appropriate for the type of monads we consider.} See \cite{Lurie2017} Section 4.7 for further details. 
	\begin{example}\label{example-monad}
		Suppose $B$ is a dg-algebra, $\cC = B-\dgmod$, and $T$ is a continuous monad acting on $\cC$. The monad $T$ can be thought of as an algebra object $A$ in the category of $B$-bimodules, or equivalently, a dg $B$-ring (i.e. a dg-algebra $A$ together with a morphism $B\to A$). The category of modules $\cC^T$ is equivalent $A-\dgmod$ (i.e. $A$-modules in $B-\dgmod$ are the same thing as $A$-modules in $\Vect$).
	\end{example}
	
	Let
	$
	F:\cD \to \cC
	$
	be a continuous functor with a left adjoint $F^L$ between compactly generated, stable, presentable $\infty$-categories. Denote by $T=FF^L$ the corresponding monad. We have $\cC^T$ the category of $T$-modules (also known as $T$-algebras) in $\cC$. Note that for any object $d\in \cD$, $F(d)$ is a module for $T$. Thus we have the following diagram:
	\[
	\xymatrix{
		\cD \ar[rr]^F \ar[rd]_{\widetilde{F}} & & \cC\\
		& \cC^T \ar[ru] &
	}
	\]
	\begin{definition}
		A functor $F: \cD \to \cC$ is called \emph{conservative} if whenever $F(x) \simeq 0$
		then $x \simeq 0$.\footnote{The usual definition of a conservative functor is a functor $F$
			such that if $F(\phi)$ is an isomorphism, then $\phi$ is an isomorphism. This definition is
			equivalent to the one above, in our context(s), by considering the cone (or the kernel and cokernel)
			of $\phi$.}
	\end{definition}
	\begin{theorem}[Barr--Beck \cite{barr_toposes_1985}, \cite{Lurie2012} Theorem 6.2.0.6, 
		\cite{Lurie2017} Theorem 4.7.3.5
		]\label{theorembarrbeck}
		The functor $\widetilde{F}:\cD \to \cC^T$ has a fully faithful left adjoint, $J$. If $F$ is conservative, then $\widetilde{F}$ and $J$ are inverse equivalences.
	\end{theorem}
	\begin{remark}\label{remark formula bb}
		\begin{enumerate}
			\item The essential image of $J$ is given by the subcategory of $\cD$ generated under colimits by the essential image of $F^L$.
			\item Given a $T$-module $c$ in $\cC$, we have a simplicial diagram in $\cD$:
			\begin{equation}\label{diagramsimplicial}
			\xymatrix{
				F^Lc & F^LFF^Lc \ar@/^/[l] \ar@/_/[l]
				& F^LFF^LFF^Lc \ar@/^/[l] \ar@/_/[l] \ar[l]
				& \ldots \ar@/^/[l] \ar@/_/[l] \ar@/_1pc/[l] \ar@/^1pc/[l]
			}
			\end{equation}
			The object $J(y)$ is given by the colimit of Diagram \ref{diagramsimplicial}. Indeed, applying the functor $F$ to diagram \ref{diagramsimplicial} is the canonical simplicial resolution of $y \in \cC$ (the bar construction). 
		\end{enumerate}
	\end{remark}
	
	\begin{example}[Koszul duality]\label{example-koszul}
	Let $V$ be a vector space, and consider the graded algebras $\Lambda = \Sym(V[1])$ and $S=\Sym(V^\ast[-2])$ (note that $\Lambda$ is an exterior algebra when considered as an ungraded algebra). Each of these algebras is equipped with an augmentation module which we denote simply by $\C$. This defines a functor
	\[
	R\bbHom_\Lambda(\C,-):\Lambda-\dgmod \to \Vect
	\]
	One can check that this functor is conservative ($\C$ is a generator) but is not continuous ($\C$ is not compact). One can fix this defect by considering the category of \emph{Ind-Coherent} $\Lambda$-modules: by definition, this is given by $\Ind(\Lambda-\Coh)$, where $\Lambda-\Coh$ is the subcategory of $\Lambda-\dgmod$ consisting of complexes with bounded cohomology, where each cohomology group is finite dimensional. By construction, $\C$ is now a compact object of $\Ind(\Lambda-\Coh)$, and we obtain a functor
	\[
	R = R\bbHom_{\Ind(\Lambda-\Coh)}(\C,-): \Ind(\Lambda-\Coh) \to \Vect
	\]
	which now satisfies the conditions of the Barr-Beck-Lurie theorem. The corresponding monad on $\Vect$ is precisely the dg-algebra $S$. Thus we obtain an equivalence
	\[
	\Ind(\Lambda-\Coh) \simeq S-\dgmod
	\]
	Restricting to the subcategory of compact objects gives the more familiar presentation
	\[
	\Lambda-\Coh \simeq S-\Perf.
	\]
	\end{example}

	\subsection{From adjunctions to recollement situations}\label{sec:localizations-and-recollements}
	A \emph{recollement situation} in $\DGCat$ is a diagram
		\[
	\xymatrixcolsep{5pc}
	\xymatrix{
		\cK \ar[r]|-{\ffi_{\ast}} & \ar@/_1pc/[l]_-{\ffi^\ast}\ar@/^1pc/[l]^-{\ffi^!} \cD \ar[r]|-{\fj^\ast} & \ar@/^1pc/[l]^-{\fj_{!}} \ar@/_1pc/[l]_-{\fj_{\ast}} \cQ 
	}
	\]
	where:
	\begin{enumerate}
			\item The functor $\fj^\ast$ is left adjoint to $\fj_\ast$ and right adjoint to $\fj_!$.
			\item The functor $\ffi_\ast$ is left adjoint to $\ffi^!$ and right adjoint to $\ffi^\ast$.
			\item The counit morphisms $\ffi^!\ffi_\ast \to \id_\cK$, $\fj^\ast \fj_\ast \to \cQ$, and the unit morphisms $\id_\cK \to \ffi^\ast \ffi_\ast$, $\id_\cQ \to \fj^\ast \fj_!$ are isomorphisms (equivalently, the functors $\fj_\ast$, $\fj_!$, and $\ffi_\ast$ are fully faithful).
			\item The essential image of $\ffi_\ast$ is the kernel of $\fj^\ast$. The essential image of $\fj_\ast$ (respectively, $\fj_!$) is the right (respectively, left) orthogonal to the essential image of $\ffi_\ast$.
			\item There are distinguished triangles of idempotent endofunctors in $\Fun(\cD,\cD)$:
			\begin{align*}
				\ffi_\ast \ffi^! \to \id_\cD \to \fj_\ast \fj^\ast \to \\
				\fj_! \fj^\ast \to \id_\cD \to \ffi_\ast \ffi^\ast \to    
			\end{align*}
	\end{enumerate}
	Given a recollement situation as above, the category $\cD$ can be reconstructed from the $\cK$ and $\cQ$ together with the functor $\ffi^\ast\fj_\ast$ (see e.g. \cite{Ayala2019}). However, we will not need this result in its generality as, in the case of interest, our recollement will be \emph{split}, in the sense of the following standard result. 
	
		\begin{proposition}\label{prop:split recollement}
		Suppose we are in the situation of Theorem \ref{theoremcoloc}. Then the following are equivalent
		\begin{enumerate}
			\item $R\bbHom (\ffi_\ast (k),\fj_!(q)) \simeq 0$ for all objects $k$ of $\cK$ and $q$ of $\cQ$; \label{p1}
			\item $\ffi^! \fj_! \simeq 0$; \label{p2}
			\item $\ffi^\ast \fj_\ast \simeq 0$; \label{p3}
			\item the natural map $\fj_! \to \fj_\ast$ is an isomorphism; \label{p4}
			\item the natural map $\ffi^! \to \ffi^\ast$ is an isomorphism; \label{p5}
		\end{enumerate} 
	\end{proposition}
		We say that the recollement is \emph{split} if any of these equivalent conditions are satisfied. In this case, every object $d$ of $\cD$ can be written as a direct sum $\ffi_\ast \ffi^!(d) \oplus \fj_\ast \fj^\ast(d)$. Moreover, there is an equivalence of stable infinity categories $\cD \simeq \cK \oplus \cQ$, where the right hand side denotes the biproduct in $\Cat$ (see e.g. \cite[Example 1.3.4]{Ayala2019}).

Suppose $F:\cD\to \cC$ is a continuous functor which admits a left adjoint $F^L$. In this subsection, we will show how $F$ induces a recollement situation on $\cD$, in which $\cD$ is glued together from the kernel of $F$ and (an appropriate completion of) the image of $F^L$.  We obtain this result by an application of the Barr-Beck-Lurie Theorem.

This result is analogous to the following idea from linear algebra.
	
\begin{example}
Suppose $f:D\to C$ is a homomorphism of finite dimensional complex vector spaces. 
We have a commutative diagram:
\[
\xymatrix{
0 \ar[r] & \ker(f)  \ar[r]^-i & D \ar[r]^-j \ar[d]_f & D/\ker(f) \ar[ld] \ar[r] & 0\\
 && C &  &
 }
\]
where the top row is exact. If in addition $C$ and $D$ are equipped with hermitian inner products, then we get a splitting of the sequence:
\[
D = \ker(f) \oplus \im(f^\dagger)
\]
where $f^\dagger$ is the adjoint of $f$. 
\end{example}

Now let us return to the categorical situation. To begin with, suppose 
\[
F:\cD \to \cC
\]
is any continuous functor in $\DGCat$ (so $F$ admits a possibly non-continuous right adjoint $F^R$). Let $\cK$ denote the kernel of $F$: the full subcategory of $\cD$ consisting of objects $d$ such that $F(d) \simeq 0$. As $F$ is continuous, the kernel $\cK$ is closed under colimits in $\cD$, or equivalently, the inclusion functor $\cK \to \cD$ is itself continuous (and thus admits a right adjoint). It is shown in \cite[Section 5]{blumberg_universal_2013} that the Verdier quotient $\cQ = \cD/\cK$, defined as cofiber of the inclusion $\cK \to \cD$ in $\DGCat$, is a Bousfield localization, i.e. the quotient map $\cD \to \cQ$ admits a fully faithful right adjoint, whose essential image is the right orthogonal to $\cK$. Thus we have a commutative diagram:
	\[
\xymatrix{
	\cK \ar[r]^{\ffi_{\ast}} &
	  \cD \ar[d] \ar[r]^{\fj^\ast} & 
	  \cQ \ar[ld] \\
	& \cC & 
}
\]
where the top row is an exact sequence in $\DGCat$ (in the sense of \cite{blumberg_universal_2013}).

The following result 
states that if we further assume $F$ to have a left adjoint, then this short exact sequence of categories may be strengthened to a recollement situation. 

	\begin{theorem}\label{theoremcoloc}
Let
$
F: \cD \too \cC
$
be a continuous functor in $\DGCat$ which admits a left adjoint $F^L$ (note that $F$ automatically admits a possibly non-continuous adjoint $F^R$). Then there is a diagram:

	\[
\xymatrixcolsep{5pc}
\xymatrix{
	\cK \ar[r]|-{\ffi_{\ast}} & \ar@/_1pc/[l]_-{\ffi^\ast}\ar@/^1pc/[l]^-{\ffi^!} \cD \ar[d]^-F \ar[r]|-{\fj^\ast} & \ar@/^1pc/[l]^-{\fj_{!}} \ar@/_1pc/[l]_-{\fj_{\ast}} \cQ \ar@/^1pc/[ld] \\
	& \cC &
}
\]
where $\ffi_\ast:\cK \to \cD$ is the inclusion of the kernel of $F$, $\fj^\ast: \cD \to \cQ = \cD/\cK$ is the Verdier quotient, and the top row of the diagram is a recollement situation. Specifically:
		\begin{enumerate}
				\item The functor $\fj^\ast$ is left adjoint to $\fj_\ast$, and right adjoint to $\fj_!$.
			\item The functors $\fj_\ast$, $\fj_!$, and $\ffi_\ast$ are fully faithful;
			\item There are distinguished triangles of idempotent endofunctors
			\begin{align*}
			\ffi_\ast \ffi^! \to \id_\cD \to \fj_\ast \fj^\ast \to \\
			\fj_! \fj^\ast \to \id_\cD \to \ffi_\ast \ffi^\ast \to    
			\end{align*}
		\end{enumerate}
Moreover, the projection functor $\fj_!\fj^\ast$ may be computed as the bar construction:
\[
\fj_!\fj^\ast = \colim (F^LF
 \leftleftarrows (F^LF) ^2 \ldots)
\]
In particular, the Verdier quotient $\cQ$ is canonically identified with the category $\cC^T$ of modules for the monad $T=FF^L$, and the splitting of $\cQ$ as the left orthogonal of $\cK$ in $\cD$ is identified with the cocompletion of the essential image of $F^L$.
	\end{theorem}
\begin{proof}
As explained in \cite{blumberg_universal_2013}, the existence of the right adjoints $\ffi^!, \fj_\ast$, and the first of the distinguished triangles in Part (3) of the theorem follows from the general theory of Bousfield localizations. 

The remainder of the claims follow from an application of the Barr-Beck-Lurie theorem. Indeed, by Theorem \ref{theorembarrbeck} there is a canonical enhancement of $F$ to a functor
\[
\widetilde{F}: \cD \to \cC^T
\]
with a fully faithful left adjoint $\widetilde{F}^L$.
By the universal property of the quotient, the functor $\widetilde{F}$ factors as follows:
\[
\xymatrix{
\cD \ar[r] \ar[rd] &\cQ \ar[d] \\ 
 & \cC^T
}
\]
Moreover, the right vertical arrow is an equivalence: its inverse is given by the composite 
\[
\xymatrix{
\cC^T \ar[r]^-{\widetilde{F}^L} &\cD \ar[r]^-{\fj^\ast} & \cQ.
}
\]
Thus we obtain the fully faithful left adjoint $\fj_!$ to $\fj^\ast$ by identifying $\cQ$ with $\cC^T$.  The functor $\ffi^\ast$ is then uniquely determined by the projection $\ffi_\ast\ffi^\ast$ which is defined as the cofiber of the map $\fj_!\fj^\ast \to \id_\cD$.
This establishes the recollement situation.
Using the identification of $\cQ$ and $\cC^T$ mentioned above, the formula for $\fj_! \fj^\ast$ follows immediately from the formula for $\widetilde{F}^L$ in Remark \ref{remark formula bb}.
\end{proof}

\begin{remark}
	With the assumptions of Theorem \ref{theoremcoloc}, there is also an equivalence $\cQ = \cC_{S}$, where $S=FF^R$ is the associated comonad. Thus there is an analogous formula for the projection $\fj_\ast \fj^\ast$ as the totalization of a cobar construction.
\end{remark}

\begin{example}
Suppose $f:X\to Y$ is an \'etale map of smooth varieties, $j: U \hookrightarrow Y$ the Zariski open image of $f$, and $i: Z\hookrightarrow Y$ its closed complement. Consider the $D$-module inverse image functor $f^!: \bD(Y) \to \bD(X)$. This gives rise to a diagram:
	\[
\xymatrix{
	\bD(Z) \ar[r]^{\ffi_{\ast}} &
	  \bD(Y) \ar[d] \ar[r]^{\fj^\ast} & 
	  \bD(U) \ar[ld] \\
	& \bD(X) & 
}
\]
However, this diagram is not a recollement situation: $f^!$ (or equivalently $j^!$) does not have a left adjoint on the entire category of $D$-modules in general. One may obtain a recollement situation by restricting to the category of ind-holonomic $D$-modules (or ind-constructible complexes of sheaves of vector spaces, say). In light of this, the existence of a recollement situation on the entire category of $D$-modules (as appears in the present paper and \cite{mcgerty_morse_2014}, for example) is somewhat remarkable.
\end{example}

	\begin{remark}
		If the functor $F$, in addition, takes compact objects in $\cD$ to compact objects in $\cC$, then the right adjoint $F^R$ preserves colimits. In that case, we can restrict the adjunction to the subcategories of compact objects, and the constructions of Theorem \ref{theoremcoloc} can be done in the category of small stable $\infty$-categories. Thus the recollement restricts to one on the level of small categories of compact objects, or equivalently, all the morphisms involved in the recollement are colimit-preserving. 
	\end{remark}
	
	\begin{remark}
	If we have an ambidextrous adjunction in the context of Theorem \ref{theoremcoloc}, that is, an isomorphism $F^L \simeq F^R$, then the induced recollement splits. Indeed, given $k\in \ker(F) = \im(\ffi_\ast)$, write $k = \colim_i k_i$ where $k_i$ are compact (and thus $\ffi_\ast(k_i)$ are compact, as $F^R\simeq F^L$ is continuous). Given a general object $q \simeq \fj^\ast(d) \in \cQ$, we have $\fj_!(q) \simeq \colim_j ((F^LF)^j(d))$. Then 
	\begin{align*}
		R\bbHom_{\cD}(\ffi_\ast (k), \fj_!(q)) &\simeq R\bbHom_{\cD}(\ffi_\ast (\colim_i k_i), \colim_j (F^L F)^j(d))\\
		&\simeq \lim_i \colim_j R\bbHom_{\cD}(\ffi_\ast (k_i), (F^L F)^j(d))\\
			&\simeq \lim_i \colim_j R\bbHom_{\cD}(\ffi_\ast (k_i), (F^R F)^j(d))\\
			&\simeq \lim_i \colim_j R\bbHom_{\cD}(F(\ffi_\ast (k_i)), F(F^L F)^{j-1}(d))\\
			&\simeq 0.
	\end{align*}
	However, ambidexterity is not a necessary condition for the associated recollement to be split as the main results of this paper demonstrate.  
\end{remark}
	%

\subsection{$D$-modules on stacks}\label{sec:d-modules-on-stacks}
In this section we record some general properties about $D$-modules that we will need in the paper.

Let $Y$ be an algebraic variety, and $K$ an affine algebraic group acting on $Y$; let $X=Y/K$ denote the quotient stack (in general, the term \emph{quotient stack} will refer to a stack that is of this form--these are the only stacks we will need to consider in this paper).

We will associate to such a stack $X$ a stable, presentable DG-category $\bD(X)$, equipped with a $t$-structure. The definition of $\bD(X)$, as well as further discussion and motivation, may be found in \cite{ben-zvi_character_2009}, Section 4. See also \cite{drinfeld_some_2013}, Section 6, and \cite{gaitsgory_study_2017a}, Section III.4.

If $X$ is a smooth variety, the (triangulated) homotopy category of $\bD(X)$ is the derived category of (the abelian category of) $D$-modules on $X$. There are many references on $D$-modules using the theory of triangulated categories setting--see, for example, \cite{hotta_d_2008} (especially Chapters 1 and 2). If $X = Y/K$ is a quotient stack, then the homotopy category of $\bD(X)$ (or rather, the analogous object in the context of constructible sheaves) was described by Bernstein--Lunts \cite{bernstein_equivariant_1994}; see also \cite[Section 7]{beilinson_quantization_} for a different perspective in the $D$-modules setting.

 The primary motivation for working in the DG context is so we may make use of the Barr-Beck-Lurie theorem (here, the DG/$\infty$-categorical enhancements are essential). Moreover, many definitions and proofs are simpler and more intuitive in the DG setting.\footnote{For example, in Section \ref{sec:localizations-and-recollements} we define a functor as the cone of a natural transformation of functors; this definition is not immediately available in the triangulated setting, as the category of functors between two triangulated categories is not itself triangulated in general.} However, for most of the paper, the reader may substitute the DG category $\bD(X)$ with its the more familiar triangulated (equivariant) derived category without losing any of the key ideas. 

The following example is useful to get a feel for how the equivariant derived category deals with stabilizer groups at a point.

The category of compact objects is denoted $\bD_{com}(X)$. Thus we have $\bD(X) = \Ind \bD_{com}(X)$. Denote by $\bD_{coh}(X)$ the category of bounded complexes each of whose cohomology objects is a coherent $D$-module. It is known that $\bD_{com}(X) \subset \bD_{coh}(X)$, with equality when $X$ is a variety, but for a non-safe stack the inclusion is strict (see Example \ref{exampleindcoh} below). 

\begin{example}[\cite{drinfeld_some_2013}, 7.2]\label{exampleBG}
		Given an affine algebraic group $K$, we have 
	\[
	\bD(pt/K) \simeq C^\ast(K)-\dgcomod \simeq C_\ast(K)-\dgmod
	\] 
	The coherent objects $\bD_{coh}(pt/K)$ correspond to the subcategory $C_\ast(K)-\Coh$ in $C_\ast(K)-\dgmod$; on the other hand, the compact objects $\bD_{com}$ correspond to $C_\ast(K)-\Perf$ (see Example \ref{example-koszul}, which applies directly to the case where $K$ is a connected reductive group).  
	
	Note that the heart of the $t$-structure on $\bD(pt/K)$ is the abelian category
	\[
	\bM(pt/K) = H_0(K)-\module = \Rep(\pi_0(K))
	\]
	 In particular, if $K$ is connected then $\bM(pt/K)$ is equivalent to the abelian category of vector spaces, and if $K$ is not contractible then $\bD(pt/K)$ is not equivalent to the derived category of $\bM(pt/K)$. On the other hand, if $K$ is unipotent, there is an equivalence of DG categories $\bD(pt/K) \simeq \Vect$.
\end{example}

The following result states that $\bD(X)$ is always compactly generated in the cases of interest in this paper.
\begin{proposition}[\cite{drinfeld_some_2013} Theorem 8.1.1]
	Let $X$ be a quotient stack.\footnote{This holds for a much broader class of stacks, e.g. QCA and perfect stacks} Then $\bD(X)$ is compactly generated by inductions from coherent sheaves on $X$.
\end{proposition}

Now let us now recall the functoriality properties of $D$-modules. We will want to consider morphisms of quotient stacks that are slightly more general than representable maps (which may just be thought of as equivariant maps of varieties). In general, the $D$-module functors may not be well-behaved, however we will restrict attention to \emph{safe} morphisms (see \cite[Definition 10.2.2]{drinfeld_some_2013}). Safety guarantees that the $D$-module pushforward is continuous. The safe morphisms appearing in this paper will always be composites of a representable morphism and a unipotent gerbe. The example to keep in mind is the projection from $pt/U \to pt$, where $U$ is a unipotent algebraic group (in fact, the pullback and pushforward functors on $D$-modules are derived equivalences here).

\begin{proposition}\label{propositiondmodulefacts} Given a safe morphism of quotient stacks $f: X \to Y$, we have continuous functors:
	\[
\xymatrix{	f_\ast : \bD(X) \ar[r]& \bD(Y),}
	\]
	\[
\xymatrix{	f^!: \bD(Y) \ar[r]& \bD(X),}
	\]
and
	\[
\xymatrix{	\D_X:\bD(X) \ar[r]^-\sim& \bD(X)'}
	\]
	(where $\bD(X)' = \Ind(\bD_{c}^{op})$ is the dual category).
	\begin{enumerate}
		\item \label{proper} If $f$ is proper, then $f_\ast \simeq \D_Y f_\ast \D_X$ preserves coherence and is right adjoint
		to $f^!$. We sometimes write $f_!$ instead of  $f_\ast$ in that case.
		\item \label{smooth} If $f$ is smooth of relative dimension $d$, then $f^!$ preserves coherence and 
		$f^\ast := f^! [-2d]$ is left adjoint to $f_\ast$.
		The functor $f^\circ = f^![-d]$ $t$-exact, and $f^\circ \simeq \D_X f^\circ \D_Y$.
		\item \label{base-change} If 
		\[
		\xymatrix{
			X \times _W V \ar[r]^-{\tilde{f}} \ar[d]_-{\tilde{g}} & V \ar[d]^-g \\
			X \ar[r]_-f & W
		}
		\]
		is a cartesian diagram of stacks, then the base change morphism is an
		isomorphism: \[
		f^! g_\ast \simeq \tilde{g}_\ast \tilde{f}^!.
		\]
		\item \label{projection} We have the projection formula:
		\[
		f_\ast \left( f^! \fM \otimes \fN \right) \simeq \fM \otimes f_\ast (\fN).
		\]
Moreover, the category $\bD(X)$ carries a symmetric monoidal tensor product 
		\[
		{\fM}\otimes  {\fN} := \Delta^! ({\fM}\boxtimes {\fN}) \simeq \fM \otimes_{\cO_X} \fN [-\dim(X)].
		\]
	\end{enumerate}
\end{proposition}
\begin{proof}
These results are now standard in the setting of triangulated categories of $D$-modules on smooth varieties. A reference in this setting is the book \cite{hotta_d_2008}: see Sections 1.5 and 2.6 for the definitions of the functors; Theorem 2.7.1, Theorem 2.7.2, Corollary 2.7.3 for parts \ref{proper} and \ref{smooth}; Theorem 1.7.3 for part \ref{base-change}; Corollary 1.7.5 for \ref{projection}.

In the setting of dg-categories of $D$-modules on stacks, a reference for these facts (for a much larger class of stacks than we require) is \cite{drinfeld_some_2013}. In particular, see Sections 6.2, 7.4, and 8.4 for the definitions of the functors; Section 7.3 for part \ref{smooth}; Theorem 10.6.2 for part \ref{proper}; Corollary 9.3.14 for \ref{base-change}; Theorem 10.2.4 for \ref{projection}.
\end{proof}

\begin{remark}
	For a non-safe morphism of stacks $f$ (e.g. the projection $pt/G \to pt$ for a reductive group $G$), to obtain a continuous functor one should take the continuous extension of the restriction of $f_\ast$ to compact objects. This is called the \emph{renormalized} pushforward, and is denoted $f_\blacktriangle$. Alternatively, one can replace $D$-modules with Ind-coherent $D$-modules as defined in Section \ref{sec:ind-coherent-d-modules}.
\end{remark}

The continuous extension of the Verdier duality functor defines a self duality of the category $\bD(X)$. The functor $f_\ast$ is dual to $f^!$ with respect to this self-duality. This gives rise to a a good theory of integral transforms for $D$-modules.
\begin{proposition}[\cite{gaitsgory_functors_2013}, 1.2; \cite{ben-zvi_character_2009}, Section 5]\label{prop-inttrans}
	Given quotient stacks $X$ and $Y$, we have an equivalence:
	\[
	\xymatrix{
		\bD(X\times Y) \ar[r]^-\sim & \ar[l] \Fun^L(\bD(X), \bD(Y)) \ar[r]^-\sim & \ar[l] \bD(X) \otimes \bD(Y). \\
		\fK \ar@{|->}[r] & \Phi_\fK := q_{Y \blacktriangle} \left( \fK \otimes q_{X}^! ( - ) \right) &.
	}
	\]
	where
	\[
	X \xleftarrow{q_X} X\times Y \xrightarrow{q_Y} Y,
	\]
	are the projections. We refer to $\fK$ as the kernel corresponding to the functor $\Phi_\fK$. 
\end{proposition}
\begin{proof}
	The identification $\Fun^L(\bD(X),\bD(Y))$ with $\bD(X\times Y)$ follows from the fact that $\bD(X)$ is self dual (\cite{drinfeld_some_2013}, Corollary 8.4.3, Corollary 8.3.4). For the formula for the functor $\Phi_\fK$ corresponding to a kernel $\fK$ see \cite{gaitsgory_functors_2013} 0.2.1.
	\end{proof} 
\begin{example}\label{example-int-kernel}
If 	
\[
X \xleftarrow{f} Z \xrightarrow{g} Y,
\]
is a diagram of smooth quotient stacks, then the functor 
$f_\ast g^!$ is represented by the integral kernel $\fK_Z = (f\times g)_{\ast} \omega_{Z}$ (this follows from the projection formula).
\end{example}

\begin{remark}\label{remark-strat}
If $Z$ has a partition $Z = \bigsqcup_{i\in I} Z_i$ in to locally closed substacks, then any object of $\bD(Z)$ on $Z$ is filtered by $I$. In that case the functor $f_\ast g^!$ (or equivalently, the object $\fK_{Z} \in \bD(X\times Y)$) has a filtration indexed by $I$ (see Definition \ref{definitionfiltration}): the functor
$(I,\leq) \to \Fun(\bD(X),\bD(Y))/(f_\ast g^!)$ is given by $i \mapsto f_{\leq i, \ast}g_{\leq i}^!$ where $f_{\leq i}$ (respectively, $g_{\leq i}$) is the restriction of $f$ (respectively, $g$) to the closed substack $Z_{\leq i} := \bigsqcup_{j\leq i} Z_j$.
\end{remark}

\begin{remark}
	The category $\bD_{hol}(X)$ of holonomic complexes is the subcategory of $\bD_{coh}(X)$ consisting of complexes whose cohomology objects are holonomic. Although coherent complexes are not always preserved by the $D$-module functors, the holonomic subcategory is preserved (for safe morphisms). By Verdier duality it follows that we have the full six functors, including adjoint pairs $(f^\ast,f_\ast)$ and $(f_!,f^!)$ on holonomic complexes for any safe morphism $f$.
\end{remark}

\subsection{Ind-Coherent $D$-modules}\label{sec:ind-coherent-d-modules}

 We define the category of \emph{Ind-coherent} $D$-modules on $X$ by
\[
\bcD(X) = \Ind \bD_{coh}(X)
\]
By construction, the compact objects of $\bcD(X)$ are exactly the coherent complexes.
\footnote{This variant of the category of $D$-modules was considered by Arinkin and Gaitsgory in the context of geometric Satake \cite{arinkin_singular_2015}, where it is referred to as the renormalized category. The construction is closely related to the theory of Ind-coherent sheaves as developed by Gaitsgory and Rozenblyum \cite{gaitsgory_ind_2011, gaitsgory_study_2017}. Note that the relationship between Ind-coherent sheaves and Ind-Coherent $D$-modules is not so clear: the usual category $\bD(X)$ is already given by $\mathrm{IndCoh(X_{dR})}$ (which is equivalent to $\QC(X_{dR})$ as the de Rham stack of anything is trivially smooth). In some sense the difference between Ind-coherent and usual $D$-modules is measuring the singularities in $T^\ast (Y/K)$ generated by the non-flatness of the moment map.} 
Both $\bD(X)$ and $\bcD(X)$ carry a $t$-structure whose heart is the (same) abelian category $\bM(X)$ of $D$-modules on $X$.
There are adjoint functors $\bD(X) \leftrightarrows \bcD(X)$ which exhibit $\bD(X)$ as a co-localization of $\bcD(X)$, which restrict to an equivalence on the positive part of the $t$-structure.
\begin{remark}
	If $X$ is a variety, or more generally if $X$ is \emph{safe} in the sense of \cite{drinfeld_some_2013}, then $\bcD(X) = \bD(X)$. 
\end{remark}

\begin{example}\label{exampleindcoh}	
Let $X = pt/T$, where $T$ is an algebraic torus, and consider $\Lambda = H_\ast(T) =\Sym( \ft[1])$, and $S=H^\ast(pt/T) = \Sym(\ft^\ast[-2])$. By Example \ref{exampleBG}, $\bD(X)$ is equivalent to $\Lambda-\dgmod$, $\bD_{com}(X) = \Lambda-\Perf$, and $\bD_{coh}(X) = \Lambda-\Coh$. As explained in Example \ref{example-koszul}, Koszul duality defines equivalence $\Lambda-\Coh \simeq S-\Perf$, and thus $\bcD(X) = S-\dgmod$.
\end{example}

\begin{remark}
	When $X$ is a finite orbit stack (for example, a quotient stack $Y/K$ where $K$ acts on $Y$ with finitely many orbits), every coherent complex on $X$ is regular holonomic. Thus, via the Riemann-Hilbert correspondence, $\bcD(X)$ can be identified with the Ind category of $K$-equivariant constructible complexes on $Y$ in the sense of Bernstein--Lunts \cite{bernstein_equivariant_1994}.
\end{remark}

Given a proper (representable) morphism of quotient stacks $f:X\to Y$, the functor $f_\ast$ preserves coherence and thus defines a continuous, quasi-proper functor $\bcD(X) \to \bcD(Y)$ which (by abuse of notation) we still write as $f_\ast$. By the remarks in Section \ref{subsectionstable}, this functor has continuous right adjoint which we denote $f^!$. Similarly, if $f$ is smooth of relative dimension $d$, then $f^!$ preserves coherent complexes. We consider the induced functor $f^\ast = f^![-2d]: \bcD(Y) \to \bcD(X)$ and write $f_\ast$ for the continuous right adjoint of $f^\ast$.

This allows us to define a pair of functors
\[
f_\ast: \bcD(X) \leftrightarrows \bcD(Y): f^!
\]
for every smooth or proper morphism of quotient stacks. One can use this to construct such functors for any representable morphism of quotient stacks by factoring such a morphism as a smooth morphism composed with a proper morphism.

The utility of Ind-Coherent $D$-modules stems from the fact that the following observation: given a  quotient stack $X$ with projection morphism $p:X\to pt$, the functor $p^!:\Vect \to \bD(X)$ does not preserve compact objects, but it does preserve coherent objects. Thus it defines a quasi-proper functor on Ind-coherent $D$-modules; we write $f^\ast = f^![-2\dim X]$, and define $f_\ast$ to be the (continuous) right adjoint of $f^\ast$. Thus a smooth quotient stack behaves just as a smooth algebraic variety from the perspective of Ind-coherent $D$-modules.
%
%
%

\begin{remark}\label{remark-ind-int-trans}
	The Verdier duality functor defines an equivalence of small stable $\infty$-categories $\bD_{coh}(X) \simeq \bD_{coh}(X)^{op}$, and thus, by taking Ind-categories, an equivalence of compactly generated stable presentable $\infty$-categories 
	\[
	\bcD(X) = \Ind \bD_{coh}(X) \simeq \Ind \bD_{coh}(X)^{op} = \bcD(X)'.
	\]
	In other words, the category $\bcD(X)$ is self-dual.
	(Recall that the dual of a compactly generated stable presentable $\infty$-category $\cC$ is given by $\cC':= \Ind (\cC_c^{op})$.)
	This gives rise to a theory of integral transforms for Ind-coherent $D$-modules as in Proposition \ref{prop-inttrans}, except one does not have to consider the renormalized pushforward, as the non-renormalized version is already continuous on Ind-coherent $D$-modules.
\end{remark}

\section{Mackey Theory and Decomposition by Levis}\label{sec:mackey-theory-and-decomposition-by-levis}
Throughout this and remaining sections, we will fix a connected, complex reductive group $G$ with Lie algebra $\fg$. We will use the same notation from \cite{Gunningham2018}, writing $\ug$ for the quotient stack $\fg/G$ (and similarly for other groups). See Section \ref{sec:notation} for an overview of the notation used in this paper.

Functors of parabolic induction and restriction were defined and studied for character sheaves by Lusztig, and a corresponding Mackey formula was proved in \cite[Proposition 15.2]{lusztig_character_1985b}. In \cite{Gunningham2018}, the author studied parabolic induction and restriction functors and proved a Mackey theorem for the abelian category of equivariant $D$-modules on $\fg$. It was also shown that induction and restriction functors induced an orthogonal decomposition of the abelian category $\bM(\ug)$. The goal of this section is to study the corresponding results for the equivariant derived category $\bD(\ug)$. 


\subsection{Parabolic induction and restriction}\label{sec:parabolic-induction-and-restriction}
Fix a parabolic subgroups $P$ with Levi factor $L\subseteq P$. Recall that we have functors: 
	\[
\xymatrixcolsep{3pc}\xymatrix{
	\IND^G_{P,L}=r_\ast s^!: \bD(\ul) \ar@/_0.3pc/[r] &  \ar@/_0.3pc/[l] \bD(\ug):  
	s_\ast r^!=\RES_{P,L}^G
}
\]
given by the diagram
\begin{equation}\label{diagramgs}
	\xymatrix{
		\ug & \ar[l]_r \up \ar[r]^s & \ul.
	}
\end{equation}
Standard properties of $D$-module functors imply that $\RES^G_{P,L}$ is right adjoint to $\IND^G_{P,L}$. We record here the properties enjoyed by these functors.

\begin{theorem}[{\cite[Corollary 3.9, Theorem 3.10]{Gunningham2018}} {\cite[Theorem 5.6]{Bezrukavnikov2021}}]
The functors $\RES^G_{P,L}$ and $\IND^G_{P,L}$ are $t$-exact and preserve coherent complexes.  
\end{theorem}

\subsection{Steinberg Stacks and Functors}\label{sec:steinberg-stacks-and-functors}
Here we recall some definitions and results from \cite{Gunningham2018}, Section 1. We fix two parabolic subgroups $P$ and $Q$ with Levi factors $L$ and $M$ containing a common maximal torus.

The fiber product $\uq \times_{\ug} \up$ will be denoted by $\St QP$ and referred to as the \emph{Steinberg stack}. It is equipped with projections 
	\[
	\xymatrix{
		\um & \ar[l]_\alpha \St QP \ar[r]^\beta & \ul.
	}
	\]
We observe that that the Steinberg stack has two incarnations as a quotient stack:
\begin{align}\label{equationsteinberg1}
\St QP &\simeq\left\{
(x,hP,kQ) \in \fg\times G/P\times G/Q \mid x\in \ls{h}{\fp} \cap \ls{k}{\fq}
\right\}/G \\\label{equationsteinberg2} 
& \simeq 	\left\{ (x,g) \in \fg \times G \mid x\in \fq \cap \ls{g}{\fp} \right\}/(Q\times P)
\end{align}
To see that these two quotient stacks are indeed equivalent, first note that Incarnation (\ref{equationsteinberg1}) is equivalent to 
\[
\left\{
(x,h,k) \in \fg\times G\times G \mid x\in \ls{h}{\fp} \cap \ls{k}{\fq}
\right\}/G\times P\times Q
\]
Then observe that $G$ acts freely on the $k$ variable, which one then eliminates to obtain Incarnation (\ref{equationsteinberg2}).

\begin{remark}The numerator in Incarnation (\ref{equationsteinberg1}) may be referred to as the (big, parabolic) Steinberg \emph{variety}, in contrast to the ``small'' Steinberg variety, which involves a nilpotent element of $\fg$ (or unipotent element of $G$). Parabolic Steinberg varieties were studied in \cite{steinberg_desingularization_1976}\footnote{Note that in \cite[Section 3]{steinberg_desingularization_1976}, the Lie algebra element $x$ is replaced by an element of a fixed unipotent orbit.}) and play a central role in Springer theory. The following results correspond to well-known facts about Steinberg varieties (in particular, see \cite[Proposition 3.3, (a),(b)]{steinberg_desingularization_1976}); however, the stacky point of view (in particular, the passage between the two incarnations (\ref{equationsteinberg1}), (\ref{equationsteinberg2}) above) provides a helpful perspective for understanding these ideas.
\end{remark}

There is a natural morphism of stacks
\[
\St QP \to \quot QGP
\]
given by the projection $(x,g) \mapsto g$ in the notation of Incarnation (\ref{equationsteinberg2}).
There are finitely many orbits of $Q\times P$ on $G$ and the closure relations define a partial order on this set (see \cite{borel_groupes_1965} Chapter 5).\footnote{If we fix a Borel subgroup with a maximal torus $T\subseteq B$, and suppose $Q$ and $P$ are contain $B$ (there is a unique choice of such $Q,P$ inside their conjugacy class), then double cosets $\quot QGP$ are in bijection with $\quot{W_Q}{W}{W_P}$ where $W_P$,$W_Q$ are the parabolic subgroups of $W=W_{G,T}$ corresponding to the standard parabolics $P$ and $Q$. The double cosets $Q\dot{v}P$ (where $\dot{v}\in N_G(T)$ is a lift of $v\in W$) are unions of usual Bruhat cells $B\dot{w}B$ for $w\in W_QvW_P$. See \cite[Corollaire 5.20]{borel_groupes_1965}, or \cite[Proposition 3.3]{steinberg_desingularization_1976}.} Thus we can pull back the orbit partition on $\quot QGP$ to define a locally closed partion of $\St QP$ indexed by the set of double cosets $\quot QGP$. For each orbit $w$ in $\quot QGP$, we denote by $\Stw QPw$ the corresponding subset in $\St QP$. In the notation of Incarnation (\ref{equationsteinberg2}), the stabilizer of the $P\times Q$ action at $(x,g)$ is $Q\cap \ls g P$. Thus for any particular representative $g=\dw$ of $w$, there is an equivalence of stacks:
		\[
		\Stw QPw \simeq (\fq \cap \ls \dw\fp) \adjquot (Q\cap \ls \dw P).
		\]

We define the \emph{Steinberg functor} 
	\[
	\ST = \lrsubsuper{M,Q}{G}{\ST}{P,L}{} := \RES^G_{Q, M} \IND^G_{P,L}: \bcD(\ul) \to \bcD(\um),
	\]
(we will often drop the subscripts and superscripts when the context is clear).
By base change, we have $\ST \simeq   \alpha_\ast \beta^!$, where:
	\[
	\xymatrix{
		&& \St QP \ar[ld] \ar[rd]  \ar@/_2pc/[lldd]_{\alpha} 
		\ar@/^2pc/[rrdd]^{\beta}&& \\
		& \uq  \ar[ld] \ar[rd]  & \Diamond & \up  \ar[ld] \ar[rd]&\\
		\um & & \ug & & \ul.
	}
	\]
By Remark \ref{remark-strat}, the functor $\ST$ has a filtration (in the sense of Definition \ref{definitionfiltration}) indexed by the poset $\quot QGP$ (we will refer to this filtration as the Mackey filtration). The following result identifies the components of the associated graded functor (see \cite{Gunningham2018}, Proposition 1.6 for the notation).

	\begin{proposition}[Mackey Filtration, \cite{Gunningham2018}, Proposition 1.6]\label{propositionmackey}
		For each lift $\dw \in G$ of $w \in \quot QGP$, there is an equivalence of functors
		\[
		\ST^w \simeq \IND _{M\cap \ls \dw P, M \cap \ls \dw L} ^{M} \RES^{\ls \dw L} _{Q\cap \ls \dw L, M \cap \ls \dw L}
		{\dw _\ast}: \bD(\ul) \to \bD(\um).
		\]
	\end{proposition}

\subsection{The Recollement by Levis}\label{sec:the-recollement-by-levis}
In \cite{Gunningham2018}, we saw how the functors of parabolic induction and restriction give rise to a recollement situation of the abelian category of $D$-modules. 

\begin{theorem}[{\cite[Theorem B]{Gunningham2018}}]
	\label{theoremabcat}
	Suppose $L$ is a Levi subgroup of $G$ and $P$ a parabolic subgroup of $G$ containing $L$ as a Levi factor. 
	\begin{enumerate}
		\item The functors 
		\[
		\xymatrixcolsep{3pc}\xymatrix{
			\IND^G_{P,L}: \bD(\ul) \ar@/_0.3pc/[r] &  \ar@/_0.3pc/[l] \bD(\ug):  
			\RES_{P,L}^G
		}
		\]
		are $t$-exact.
		\item The restriction of the adjunction to the heart is a bi-adjunction
		\[
		\xymatrixcolsep{3pc}\xymatrix{
			\ind^G_{P,L}: \bM(\ul) \ar@/_0.3pc/[r] &  \ar@/_0.3pc/[l] \bM(\ug):  
			\res_{P,L}^G
		}
		\]
		\item The functors $\ind^G_{P,L}$ and $\res^G_{P,L}$ are independent of the choice of parabolic $P$ containing the given Levi subgroup $L$. 
	\end{enumerate}
\end{theorem}

\begin{corollary}\label{corindep}
	Let $L$ be a Levi subgroup of $G$ and $P,P'$ parabolic subgroups of $G$ containing $L$ as a Levi factor. Then for any $\fM \in \bD(\ug)$, $\RES^G_{P,L}(\fM) \simeq 0$ if and only if $\RES^G_{P',L}(\fM) \simeq 0$. 
\end{corollary}
\begin{proof}
Let $\fM \in \bD(\ug)$, $L$ a Levi subgroup of $G$ and $P$ a parabolic subgroup containing $L$ as a Levi factor. By non-degeneracy of the $t$-structure on $\bD(\ug)$, $\RES^G_{P,L}(\fM)\simeq 0$ if and only if $H^i(\RES^G_{P,L}(\fM))\cong 0$ for all $i\in \Z$. By $t$-exactness of $\RES^G_{P,L}$, $H^i(\RES^G_{P,L}(\fM)) \cong \RES^G_{P,L}(H^i(\fM))$. But the condition that the $\RES^G_{P,L}(H^i(\fM))\cong 0$ is independent of $P$ by Theorem \ref{theoremabcat} as required. 
\end{proof}

\begin{remark}
	Unlike the abelian category situation, it is not clear that there is an isomorphism of functors $\RES^G_{P,L} \simeq \RES^G_{P',L}$ in general (and similarly for induction). 
\end{remark}
\begin{definition}
	Let $L$ be a Levi subgroup of $G$.
	\begin{enumerate}
		\item  We write $\bD(\ug)_{\nleq (L)}$ for the kernel of the parabolic restriction functor; that is, the full subcategory of $\bD(\ug)$ consisting of objects $\fM$ such that $\RES^G_{P,L}(\fM) \simeq 0$ for some parabolic subgroup $P$ of $G$ containing $L$ as a Levi factor. By Corollary \ref{corindep}, this subcategory is independent of the choice of $P$. Thus it depends only on the conjugacy class $(L) \in \Levi_G$ of the Levi subgroup. 
		\item  Similarly, we write $\bD(\ug)_{\nless(L)}$ for the intersection of $\bD(\ug)_{\nleq (L')}$ as $L'$ ranges over Levi subgroups of $G$ with $(L')<(L)$. In other words, objects of $\bD(\ug)_{\nless(L)}$ are killed by parabolic restriction for any Levi $L'$ with $(L')<(L)$.
		\item We write $\bD(\ug)_{cusp}$ for $\bD(\ug)_{\nless(G)}$. Objects of $\bD(\ug)_{cusp}$ are called \emph{cuspidal}. In other words, cuspidal objects are killed by parabolic restriction for any proper Levi subgroup of $G$.
		\end{enumerate}
	\end{definition}

Consider the functor
\[
\xymatrix{
	\RES^G_{P,L}|_{\bD(\ug)_{\nless(L)}}: \bD(\ug)_{\nless (L)} \ar[r]& \bD(\ul)_{cusp}
}
\]
Note that the kernel of $\RES^G_{P,L}|_{\nless(L)}$ is precisely $\bD(\ug)_{\nleq(L)}$ and has left adjoint $\IND^G_{P,L}|_{\bD(\ul)_{cusp}}$. This is precisely the context of Theorem \ref{theoremcoloc}, which yields the following result.
\begin{proposition}\label{propositionrecollement}
	There is a recollement situation:
	\[
	\xymatrixcolsep{5pc}
	\xymatrix{
		\bD(\ug)_{\nleq(L)} \ar[r]|-{\ffi_{\nleq(L)\ast}} & \ar@/_1pc/[l]_-{\ffi_{\nleq (L)}^\ast}\ar@/^1pc/[l]^-{\ffi_{\nleq(L)}^!} \bD(\ug)_{\nless(L)} \ar[d] \ar[r]|-{\fj_{(L)}^\ast} & \ar@/^1pc/[l]^-{\fj_{(L)!}} \ar@/_1pc/[l]_-{\fj_{(L)\ast}} \bD(\ug)_{(L)} \ar@/^1pc/[ld] \\
		& \bD(\ul)_{cusp} &
	}
	\]
\end{proposition} 

\begin{remark}
	In the terminology of \cite{Ayala2019}, the category $\bD(\ug)$ is \emph{stratified} by the poset $\Levi_G$. 
\end{remark}

\begin{convention}
	In what follows, we will always identify the subquotient category $\bD(\ug)_{(L)}$ with the full subcategory of $\bD(\ug)$ given by the essential image of $\fj_{(L)!}$, which is the left orthogonal to $\bD(\ug)_{\nleq(L)}$ in $\bD(\ug)_{\nless (L)}$, or equivalently, the cocompletion of the essential image of $\IND^G_{P,L}\mid_{cusp}$.
\end{convention}

Explicitly, Proposition \ref{propositionrecollement} means that we can express any object $\fM \in \bD(\ug)$ as an iterated extension of objects taken from $\bD(\ug)_{(L)}$ as $L$ varies over Levi subgroups of $G$. This can be seen by the following algorithm:
\begin{enumerate}
	\item First choose a Levi subgroup $L$ which is minimal such that $\RES^G_{P,L}(\fM) \nsimeq 0$ (thus $\fM \in \bD(\ug)_{\nless(L)}$). 
	\item We have a distinguished triangle:
	\[
	\ffi_{\nleq(L)\ast} \ffi_{\nleq(L)}^\ast \fM \to \fM \to \fj_{(L)!} \fj_{(L)}^\ast \fM \xrightarrow{+1},
	\]
	\item Replace $\fM$ by $\ffi_{\nleq (L)\ast} \ffi_{\nleq (L)}^\ast(\fM)$ and repeat steps (1) and (2) (the algorithm halts after finitely many steps when $L=G$).
\end{enumerate}

\subsection{The recollement is split}
The functors of induction and restriction for abelian categories induce a recollement at the level of abelian categories, analogous to that in Proposition \ref{propositionrecollement}.
\begin{equation}
	\label{eq:aberecollement}
	\xymatrixcolsep{5pc}
	\xymatrix{
		\bM(\ug)_{\nleq(L)} \ar[r]|-{\ffi_{\nleq(L)\ast}} & \ar@/_1pc/[l]_-{\ffi_{\nleq (L)}^\ast}\ar@/^1pc/[l]^-{\ffi_{\nleq(L)}^!} \bM(\ug)_{\nless(L)} \ar[r]|-{\fj_{(L)}^\ast} & \ar@/^1pc/[l]^-{\fj_{(L)!}} \ar@/_1pc/[l]_-{\fj_{(L)\ast}} \bM(\ug)_{(L)}
	}
\end{equation}
Moreover, all the functors in the recollement above are exact. It follows that the recollement situations at the derived and abelian level are compatible in the strongest sense: all the functors in the derived recollement are $t$-exact, and restrict to the corresponding functors in the abelian recollement. 
The following result is a straightforward consequence of the Mackey formula.
\begin{proposition}[\cite{Gunningham2018} Proposition 4.17]\label{propositionorthogonal}
	Suppose $L$ and $M$ are non-conjugate Levi subgroups, and let $P$ (respectively $Q$) be parabolic subgroups containing $L$ (respectively $M$) as Levi factors. Given $\fN \in \bD_{coh}(\ul)_{cusp}$, and $\fM \in \bD_{coh}(\um)_{cusp}$ we have that $\IND^G_{P,L}(\fN)$ and $\IND^G_{Q,M}(\fM)$ are orthogonal, that is,
	\[
	R\bbHom(\IND^G_{P,L}(\fN), \IND^G_{Q,M}(\fM)) \simeq R\bbHom(\IND^G_{Q,M}(\fM),\IND^G_{P,L}(\fN)) \simeq 0.
	\]
\end{proposition}

\begin{proposition}\label{prop:abcatsplit}
	Suppose $L$ and $M$ are non-conjugate Levi subgroups and let $\fM \in \bM(\ug)_{(L)}$, $\fN \in \bM(\ug)_{(M)}$. Then
	\[
	R\bbHom_{\bD(\ug)}(\fM,\fN) \simeq 0.
	\]
\end{proposition}
\begin{proof}
By \cite[Lemma 4.30]{Gunningham2018}, $\fM$ is a direct summand of $\ind^G_L(\fM')$ for some object $\fM' \in \bM(\ul)_{cusp}$, and similarly for $\fN$. Thus the result follows from Proposition \ref{propositionorthogonal}.
\end{proof}

We can now prove the following key result. 
\begin{theorem}\label{thm: recollement is split}
	The recollement of Proposition \ref{propositionrecollement} is split for each Levi subgroup $L$ of $G$. 
\end{theorem}
\begin{proof}
	Let $L$ be a Levi subgroup of $G$. Let $\fM \in \bD(\ug)_{\nleq (L)}$ and $\fN \in \bD(\ug)_{(L)}$. We must show that
	\[
	R\bbHom(\fM,\fN) \simeq 0.
	\]
	By Lemma \ref{prop:abcatsplit}, the statement is true if $\fM, \fN$ are in the heart of the $t$-structure (note that $\fM$ is a direct sum of objects in $\bM(\ug)_{(L')}$ as $(L')$ ranges over $\nleq(L)$ in $\Levi_G$). It follows that the statement remains true for coherent complexes, which are obtained from those in the heart of the $t$-structure by shifts and cones. In particular, the statement is true for compact objects $\fM$ and $\fN$. 

If $\fM$ and $\fN$ are arbitrary objects as in the statement of the theorem, they can be written as colimits of compact objects: $\fM = \colim \fM_i$ and $\fN = \colim \fN_j$ (note that the inclusion functors $\fj_{(L)!}$ and $\ffi_{\nleq (L)}$ are left adjoints, and thus these colimits may be taken in the ambient category or the respective subcategories). We then have
\begin{align*}
R\bbHom(\fM, \fN) &\simeq R\bbHom(\colim_i \fM_i, \colim_j \fN_j)\\
&\simeq \lim_i R\bbHom(\fM_i, \colim_j \fN_j)\\
&\simeq \lim_i \colim_j R\bbHom(\fM_i, \fN_j)\\
&\simeq 0 
\end{align*}
as required. 
\end{proof}

Putting together these recollements for each $(L) \in \Levi_G$, we see that the stratification of $\bD(\ug)$ indexed by $\Levi_G$ is split: every object $\fM \in \bD(\ug)$ is a direct sum $\bigoplus_{(L) \in \Levi_G} \fM_{(L)}$ where $\fM_{(L)} \in \bD(\ug)_{(L)}$. Thus, we have the following result. 

\begin{corollary}\label{corollaryod}
There is a direct sum decomposition:
\[
\bD(\ug) = \bigoplus_{(L) \in \Levi_G} \bD(\ug)_{(L)}.
\]
\end{corollary}

\section{Generalized Springer Decomposition}\label{sec:generalized-springer-decomposition}
The goal of this section is to study the block decomposition of the subcategory of cuspidal objects. These results will then be combined with those of Section \ref{sec:mackey-theory-and-decomposition-by-levis} to prove Theorem \ref{maintheorem decomp}. The results of this section make essential use of key properties of cuspidal local systems due to Lusztig \cite{lusztig_intersection_1984}; most notably, the fact that cuspidal sheaves have distinct central characters.
	
\subsection{Abelian category decomposition}
A \emph{cuspidal local system} is, by definition, a $G$-equivariant local system $\cE$ on a nilpotent orbit $\cO \subseteq \cN_G$ such that $(\overline{\cO} \hookrightarrow \fg)_\ast IC(\cO;\cE)$ is a cuspidal object in $\bM(\ug)$. A simple cuspidal local system determines a simple object $\fE := (\overline{\cO} \hookrightarrow \cN_G)_\ast IC(\cO;\cE)$ of $\bM(\ucN_G)$ and a full subcategory $\bM(\ug)_{(\cE)}$ consisting of all objects of the form $(\fz(\fg) \times \cN_G) \hookrightarrow \fg)_\ast(\fN \boxtimes \fE)$ where $\fN$ is any object of $\bM(\fz(\fg))$. As shown in \cite[Proposition 4.22]{Gunningham2018}, the subcategory of cuspidal objects $\bM(\ug)_{cusp}$ decomposes as a direct sum of blocks $\bM(\ug)_{(\cE)}$ (the derived analogue of this result is shown in Proposition \ref{propositionodcusp} below).

A \emph{cuspidal datum} for $G$ is a pair $(L,\cE)$, where $L$ is a Levi subgroup of $G$ and $\cE$ is a simple cuspidal local system on a nilpotent orbit of $L$, or equivalently, a simple cuspidal object of $\bM(\ucN_L)$). 

Let $\bM(\ug)_{(L,\cE)}$ denote the subcategory of $\bM(\ug)$ consisting of direct summands of objects of the form $\ind^G_L(\fM)$, where $\fM \in \bM(\fl)_{(\cE)}$. The following is the main result of \cite{Gunningham2018}. 
\begin{theorem}[\cite{Gunningham2018} Theorem A]\label{theoremabelian}
	There is an orthogonal decomposition:
	\[
	\bM(\ug) \simeq \bigoplus_{(L,\cE)} \bM(\ug)_{(L,\cE)},
	\]
	where $(L,\cE)$ ranges over the set of cuspidal data for $G$. Moreover, for each cuspidal datum $(L,\cE)$, there is an equivalence of categories:
	\[
	\bM(\ug)_{(L,\cE)} \simeq \bM(\fz(\fl))^{W_{G,L}}.
	\] 
\end{theorem}

	\subsection{The derived category of cuspidal objects}\label{sec:cusp}
	Let $x\in \cN_G$ be a nilpotent element, and $\cO$ the corresponding nilpotent orbit. As usual, we write $\ucO$ for $\cO/G$, which is equivalent to the classifying stack $pt/Z_G(x)$. We write $Z^\circ(G)$ for the neutral component of the center of $G$. Let $A = A_G(x) = \pi_0(Z_G(x)) = Z_G(x)/Z_G(x)^\circ$ be the equivariant fundamental group of $\cO$. Representations of $A$ (or modules for $\C[A]$) are the same thing as $G$-equivariant local systems on $\cO$. More generally, we have the following description of derived local systems: $\bD(\ucO) \simeq C_\ast(Z_G(x))-\dgmod$ (see Example \ref{exampleBG}). 
	
	Recall that if the orbit $\cO$ supports a cuspidal local system $\cE$, then $\cO$ must be \emph{distinguished}, i.e. $Z^\circ(G)$ is a a maximal connected reductive subgroup of $Z_G(x)$ \cite[Proposition 2.8]{lusztig_intersection_1984} (see also \cite[Proposition 2.6]{achar_modular_2017}). The following lemma describes the derived local systems on a distinguished orbit. We denote by $\Lambda_{\fz(\fg)}$ the free graded-commutative algebra generated by $\fz(\fg)[1]$ (considered as a dg-algebra with zero differential). There is a quasi-isomorphism of dg-algebras:
	\[
	C_\ast(Z^\circ(G)) \simeq \Lambda_{\fz(\fg)}.
	\]
	
		\begin{lemma}
		Suppose $x$ is distinguished. Then there is a quasi-isomorphism of dg-algebras: 
		\[
		C_\ast(Z_G(x)) \simeq \Lambda_{\fz(\fg)} \otimes \C[A].
		\]
	\end{lemma}
	\begin{proof}
		Recall that by Jacobson-Morazov theory, there is an $\fsl_2$-triple, $\phi:\fsl_2 \to \fg$ such that $\phi(e) = x$ (where $e,f,h$ are the standard basis for $\fsl_2$). Moreover, we have extensions
		\[
		\xymatrix{
			Z_G(x)^\circ \ar[r] & Z_G(x) \ar[r] & A\\
			Z_G(\phi)^\circ \ar[u] \ar[r] & Z_G(\phi)	\ar[u] \ar[r] & A \ar[u]
		}
		\]
		where $Z_G(\phi)^\circ$ is the maximal reductive subgroup of $Z_G(x)^\circ$. In particular, the vertical maps are all homotopy equivalences, and thus $C_\ast(Z_G(x)) \simeq C_\ast(Z_G(\phi))$. As $x$ is distinguished, $Z_G(\phi)^\circ = Z^\circ(G)$, so the lower line is in fact a central extension of the finite group $A_G(x)$ by the torus $Z^\circ(G)$. As such, it is determined by a cocycle 
		\[
		\mu:A\times A \to Z^\circ
		\]
		Explicitly, $Z_G(\phi)$ is isomorphic to an algebraic group, whose underlying variety is $Z^\circ(G) \times A$, but with group structure twisted by the cocycle:
		\[
		(z_1,a_1)\cdot_{\mu}(z_2,a_2) = (z_1z_2\mu(a_1,a_2),a_1a_2).
		\]
		By the K\"unneth theorem, $C_\ast(Z_G(\phi))$ is equivalent as a chain complex to $\Lambda_{\fz(\fg)} \otimes \C[A]$. The convolution operation is determined by the induced map 
		\[
		\mu_\ast: \C[A]\otimes \C[A] \to \Lambda_{\fz(\fg)}
		\]
		which is necessarily trivial, as $Z^\circ(G)$ is connected and thus the induced map at the level of $H_0$ is trivial. This gives the required quasi-isomorphism.
	\end{proof}
Thus we have:
\[
\bD(\ucO) \simeq \Rep(A) \otimes \Lambda_{\fz(\fg)}-\dgmod
\]
It follows that the category of $D$-modules on $\ucO$ decomposes over the the set $\widehat{A}$ of irreducible representations of $A$.
\begin{lemma}\label{lemmafinitegp}
	There is an orthogonal decomposition:
\[
\bD(\ucO) \simeq \bigoplus_{i \in \widehat{A}}^\perp \Lambda_{\fz(\fg)}-\dgmod
\]
\end{lemma} 
By Example \ref{exampleindcoh}, we have the corresponding result for Ind-coherent $D$-modules:
\[
\bcD(\ucO) \simeq \bigoplus_{i \in \widehat{A}} S_{\fz(\fg)^\ast}-\dgmod
\]
where $S_{\fz(\fg)^\ast} = \Sym(\fz(\fg)^\ast[-2])$.
 
\subsection{Cleanness of cuspidal local systems}
Given an equivariant local system $\cE$ on $\cO$, we say $\cE$ \emph{has a clean extension to $G$} (or simply, \emph{is clean}) if the canonical maps $j_!(\cE) \to j_{!\ast}(\cE) \to j_!(\cE)$ are equivalences, where $j: \ucO \hookrightarrow \ug$ is the (locally closed) inclusion. 

Lusztig has shown that all cuspidal local systems are clean \cite[23.1]{lusztig_character_1986} using the fact that non-isomorphic cuspidal sheaves have distinct central characters (which follows from the case-by-case classification of cuspidal sheaves in \cite{lusztig_intersection_1984}). We were unable to give an independent proof of this fact using the results we have proved so far.
The orthogonal decomposition by Levis, does at least give us this result (also found in \cite{rider_perverse_2016}, Proposition 4.2):
\begin{proposition}\label{prop-cleanness}
	The following are equivalent:
	\begin{enumerate}
		\item All simple cuspidal local systems are clean.
		\item Any two non-isomorphic simple cuspidal objects of $\bM(\ucN_G)$ are orthogonal in $\bD(\ug)$.
	\end{enumerate}
\end{proposition}
\begin{proof}
	Suppose that all cuspidal local systems are clean, and let $\cE_1, \cE_2$ be nonisomorphic simple cuspidal local systems supported on nilpotent orbits $\cO_1,\cO_2$ respectively. We will denote by $\fE_1, \fE_2$ the clean extension to (simple) objects of $\bM(\ug)$. We have:
	\[
	R\bbHom(\fE_1, \fE_2) \simeq R\bbHom_{\bD(\ug)}(j_{1!} \cE_1, j_{2\ast} \cE_2) \simeq R\bbHom_{\bD(
		\ucO_2)}(j_2^\ast j_{1!} \cE_1,\cE_2).
	\]
	If $\cO_1 \neq \cO_2$, then $j_2^\ast j_{1!} \cE_1 \simeq 0$; if $\cO_1=\cO_2$, then $j_2^\ast j_{1!} \cE_1 \simeq \cE_1$, and the right hand side again vanishes by Lemma \ref{lemmafinitegp}.
	
	Now suppose that any two simple cuspidal objects of $\bM(\ucN_G)$ are orthogonal in the derived category. Let $\cE$ be a simple cuspidal local system on a nilpotent orbit $j:\cO \hookrightarrow \cN_G$. If $\cE$ were not clean, then $i^\ast j_\ast \cE \nsimeq 0$ for some other nilpotent orbit $i: \cO' \hookrightarrow \cN_G$ in the closure of $\cO$. Thus we have that 
	\[
	R\bbHom(j_\ast \cE_1,i_\ast \cF) \nsimeq 0,
	\]
	where $\cF = i^\ast j_\ast \cE \in \bD(\ucO')$. Now consider $i_\ast \cF \in \bD(\ucN_G)$ (which we identify as a subcategory of $\bD(\ug)$ as usual). Consider the decomposition $i_\ast\cF = \bigoplus (i_\ast\cF)_{(L)}$ given by Theorem \ref{thm: recollement is split}. But there is a non-zero map 
	\[
	j_\ast\cE \to i_\ast(\cF)= \bigoplus (i_\ast \cF)_{(L)}.
	\]
	As $j_\ast\cE$ is cuspidal (i.e. in the subcategory $\bD(\ug)_{(G)}$), and the decomposition is orthogonal, we see that the cuspidal summand $(i_\ast\cF)_{(G)}$ must be non-zero and is not orthogonal to $j_\ast\cE$. This contradicts the assumption that any two simple cuspidal objects are orthogonal. Thus $\cE$ must be clean, as required.
\end{proof}
	
	\subsection{Cuspidal Blocks}
From this point on we will freely use the fact that cuspidal local systems are clean. 
	
	Suppose $\cE$ is a simple cuspidal local system on $\cO$, corresponding to a certain irreducible representation of $A$. Let $\bD(\ucO)_{\cE} \simeq \Lambda_{\fz(\fg)}-\dgmod$ denote the block of $\bD(\ucO)$ corresponding to $\cE$. Let $\bD(\fz(\fg)\times \ucO)_{(\cE)}$ denote the corresponding subcategory  of $\bD(\fz(\fg) \times \ucO)$. Note that $\ug \simeq \fz(\fg) \times \fg'/G$, where $\fg' = [\fg,\fg]$.
	Consider the map: 
	\[
	k:\fz(\fg) \times \ucO \hookrightarrow \ug
	\]
	
	\begin{lemma}\label{lemma-ff}
	The functor $k_\ast: \bD(\ucO \times \fz(\fg))_{\cE} \to \bD(\ug)$ is fully faithful.
	\end{lemma}
\begin{proof}
This follows immediately from the fact that objects of $\bD(\ucO \times \fz(\fg))_{(\cE)}$ have clean extensions to $\ug$.
\end{proof}	

Denote the essential image of $k_\ast: \bD(\ucO \times \fz(\fg))_{\cE} \to \bD(\ug)$ by $\bD(\fg)_{(\cE)}$. In other words, these objects are of the form $\cE \boxtimes \fM$, where $\fM \in \bD(\fz(\fg))$ (and we identify $\cE$ with its clean extension to $\ucN_G$). We have the following description of the category of cuspidal objects in $\bD(\ug)$.
	\begin{proposition}\label{propositionodcusp}
		We have an orthogonal decomposition 
		\[
		\bD(\ug)_{cusp} \simeq \bigoplus_{(\cE)} \bD(\ug)_{(\cE)}.
		\]
		Moreover, for each cuspidal local system $\cE$, there is an equivalence of categories:
		\[
		\bD(\ug)_{(\cE)} \simeq \bD(\fz(\fg)/Z^\circ(G)) \left(\simeq \fD_{\fz(\fg)}\otimes \Lambda_{\fz(\fg)}-\dgmod\right).
		\] 
	\end{proposition}
\begin{proof}
By Lemma \ref{lemma-ff}, the category $\bD(\ug)_{cusp}$ decomposes as a direct sum indexed by distinguished orbits, where each summand is equivalent to the subcategory generated by cuspidal objects in $\bD(\ucO\times \fz(\fg))$. The result then follows from Lemma \ref{lemmafinitegp}.
\end{proof}

\begin{remark}
The subcategory $\bD_{coh}(\fg)_{(\cE)}$ consists precisely of those objects with finitely many nonzero cohomology modules, all of which are contained in $\bM_{coh}(\fg)_{(\cE)}$.
\end{remark}

	\subsection{Proof of Theorem \ref{maintheorem decomp}}
	We have already shown in Corollary \ref{corollaryod} that there is an orthogonal decomposition:
	\[
    \bD(\ug) = \bigoplus_{(L)} \bD(\ug)_{(L)}.
    \]
	Now let us fix a Levi subgroup $L$. Applying Proposition \ref{propositionodcusp} (with $L$ in place of $G$) we obtain another orthogonal decomposition:
	\[
	\bD(\ul)_{cusp} = \bigoplus_{(\cE)} \bD(\ul)_{(\cE)}.
	\]
	Let $\bD(\ug)_{(L,\cE)}$ denote the subcategory of $\bD(\ug)_{(L)}$ consisting of complexes $\fM$ such that $\RES^G_{(P,L)}(\fM) \in \bD(\ul)_{(\cE)}$.
	
	The following lemma is a consequence of the fact that the relative Weyl group acts trivially on the set of isomorphism classes of simple cuspidal objects (see \cite[Theorem 9.2 (b)]{lusztig_intersection_1984}).
	\begin{lemma}\label{lemma-adjunction-E}
The functors of parabolic induction and restriction restrict to form a monadic adjunction:
	\[
	\xymatrixcolsep{5pc}
	\xymatrix{
		\IND^G_{P,L}: \bD(\ul)_{(\cE)} \ar@/_0.3pc/[r] &  \ar@/_0.3pc/[l] \bD(\ug)_{(L,\cE)}:  
		\RES_{P,L}^G
	} 
	\]
	\end{lemma}
\begin{proof}
First note that the adjunction
\[
\xymatrixcolsep{5pc}
\xymatrix{
	\IND^G_{P,L}: \bD(\ul)_{cusp} \ar@/_0.3pc/[r] &  \ar@/_0.3pc/[l] \bD(\ug)_{(L)}:  
	\RES_{P,L}^G
} 
\]
is monadic by the general setup described in Theorem \ref{theoremcoloc}.
It remains to check that the restriction of $\IND^G_{P,L}$ to $\bD(\ul)_{(\cE)}$ has essential image contained in $\bD(\ug)_{(L,\cE)}$. In other words,we must check that for $\fN \in \bD(\ul)_{(\cE)}$, we have 
\[\lrsubsuper{P,L}{G}{\ST}{P,L}{}(\fN) :=\RES^G_{P,L}\IND^G_{P,L}(\fN) \in \bD(\ul)_{(\cE)}
\]
By the Mackey formula (Proposition \ref{propositionmackey}), $\lrsubsuper{M,Q}{G}{\ST}{P,L}{}(\fN)$ is an iterated extension of objects $w_\ast(\fN)$, as $w$ ranges over the relative Weyl group $W_{G,L}$ (here we use crucially that $\fN$ is cuspidal). 
It follows from \cite[Theorem 9.2 (b)]{lusztig_intersection_1984} that we have $w_\ast(\cE)\simeq \cE$ for all $w\in W_{G,L}$ . Thus $w_\ast(\fN) \in \bD(\ul)_{(\cE)}$ for all $w\in W_{G,L}$, and hence $\lrsubsuper{M,Q}{G}{\ST}{P,L}{}(\fN)\in \bD(\ul)_{(\cE)}$ as required.
	\end{proof}

\begin{lemma}\label{lemma-od-(L,E)}
	The subcategories $\bD(\ug)_{(L,\cE)}$ are orthogonal for nonisomorphic simple cuspidal local systems $\cE$ for $L$.
\end{lemma}
\begin{proof}
	Given two non-isomorphic simple cuspidal local systems $\cE_1, \cE_2$ for $L$, let $\fM_1 \in \bD(\ug)_{(L,\cE_1)}$ and $\fM_2 \in \bD(\ug){(L,\cE_2)}$. We must show that $R\bbHom_{\bD(\ug)}(\fM_1,\fM_2) \simeq 0$. First we note that, by Lemma \ref{lemma-adjunction-E}, $\fM_1$ can be written as a colimit of objects of the form $\IND^G_{P,L}(\fN_1)$ where $\fN_1 \in \bD(\ul)_{(\cE_1)}$. Thus it suffices to show that $R\bbHom_{\bD(\ug)}(\IND^G_{P,L}(\fN_1),\fM_2)\simeq 0$ for any object $\fN_1 \in \bD(\ul)_{(\cE_1)}$. Next
	we note that $\RES^G_{P,L}(\fM_2) \in \bD(\ul)_{(\cE_2)}$ (again by Lemma \ref{lemma-adjunction-E}), which is orthogonal to $\bD(\ul)_{(\cE_1)}$ by the assumption that $\cE_1$ is not isomorphic to $\cE_2$.
	It follows that
	\[
	R\bbHom_{\bD(\ug)}(\IND^G_{P,L}(\fN_1),\fM_2) \simeq 	R\bbHom_{\bD(\ug)}(\fN_1,\RES^G_{P,L}(\fM_2)) \simeq 0,
	\]
	for any object $\fN_1 \in \bD(\ul)_{(\cE_1)}$, as required. 
\end{proof}

Now, for each simple cuspidal local system $\cE$ for $L$, we may consider the composite of parabolic restriction $\RES^G_{P,L}$ followed by projection on to the factor $\bD(\ul)_{(\cE)}$. This establishes a recollement situation as in Section \ref{sec:the-recollement-by-levis}, where the corresponding quotient category is identified with $\bD(\ug)_{(L,\cE)}$. By Lemma \ref{lemma-od-(L,E)}, the subcategories $\bD(\ug)_{(L,\cE)}$ are orthogonal, giving an orthogonal decomposition of $\bD(\ug)_{(L)}$. Putting this all together, we obtain the orthogonal decomposition:
\[
\bD(\ug) = \bigoplus_{(L,\cE)} \bD(\ug)_{(L,\cE)}.
\]
This concludes the proof of Theorem \ref{maintheorem decomp}.	
	\section{The Steinberg Monad}\label{sec:the-steinberg-monad}
In this section we study the blocks $\bD(\ug)_{(L,\cE)}$ via parabolic restriction, and the corresponding Mackey filtration on the monad. Conceptually, the Steinberg monads appearing in this paper and in \cite{Gunningham2018} replace the endomorphism (or $\Ext$) algebra of the Springer sheaf or equivalently the homology of the Steinberg variety appearing in previous work on (generalized) Springer theory (e.g. \cite{borho_representations_1981,chriss_representation_1997, lusztig_intersection_1984}). In particular, we will prove Theorem \ref{maintheorem monad} and Theorem \ref{maintheorem dg algebra}.
	
	\subsection{Monadic description of the blocks}
		Fix cuspidal data $(L,\cE)$, and let $\uz(\fl) = \fz(\fl)/Z^\circ(L) \simeq \fz(\fl) \times pt/Z^\circ(L)$. By Lemma \ref{lemma-adjunction-E} the functors of parabolic induction and restriction restrict to form a monadic adjunction:
	\[
	\xymatrixcolsep{5pc}
	\xymatrix{
		\IND^G_{P,L}: \bD(\ul)_{(\cE)} \ar@/_0.3pc/[r] &  \ar@/_0.3pc/[l] \bD(\ug)_{(L,\cE)}:  
		\RES_{P,L}^G
	} 
	\]
	According to Proposition \ref{propositionodcusp}, $\bD(\ul)_{(\cE)}$ is equivalent to $\bD(\uz(\fl))$. Let us denote by $\ST = \ST_{\cE}$ the corresponding monad acting on $\bD(\uz(\fl))$ under this equivalence. Thus, by the Barr-Beck-Lurie Theorem \ref{theorembarrbeck} there is an equivalence 
	\[
	\bD(\ug)_{(L,\cE)} \xrightarrow{\sim} \bD(\uz(\fl))^{\ST}.
	\]
	This establishes part \ref{part monad} of Theorem \ref{maintheorem monad}. 
	
	\subsection{Double cosets and the relative Weyl group}\label{remark double coset}
	In this section we will recall the relationship between the relative Weyl group $W_{G,L}$ and the poset $\quot PGP$.
	
	Along let us fix a pair $P,L$ of a parabolic subgroup $P$ and Levi factor $L$, along with a compatible choice of maximal torus $T\subseteq L$ and Borel subgroup $B\subseteq P$. Let $W=W_{G,T} = N_G(T)/T$ and $W_L = W_{L,T}$ denote the usual Weyl groups of $G$ and $L$ respectively, equipped with a Coxeter structure via the choice of Borel $B$. Throughout this discussion, given an element $w\in W = N_G(T)/T$, $\dot{w} \in N_G(T)$ denotes a lift of $w$. Let us first recall how to relate the posets $\quot PGP$ and the relative Weyl group $W_{G,L}$ in terms of the Coxeter groups $W$ and $W_L$.
	
	The usual Bruhat decomposition exhibits an equivalence of posets $W \cong \quot BGB$ via $w \mapsto B\dot{w}B$. Similarly, the parabolic Bruhat decomposition exhibits an isomorphism of posets 
	\[
	\quot PGP \cong \quot{W_L}{W}{W_L}.
	\]
	More explicitly, every double coset $PgP$ in $\quot PGP$ has a representative $P\cdot{w}P$ with $\dot{w} \in N_G(T)$, and the corresponding double coset $W_LwW_L$ in $\quot{W_L}{W}{W_L}$ is uniquely determined. Similarly, one may check directly that the map $wW_L \mapsto \dot{w}L$ defines an isomorphism of finite groups $N_W(W_L)/W_L \xrightarrow{\sim} N_G(L)/L = W_{G,L}$. Finally, we note that the natural map $N_W(W_L) \to \quot{W_L}{W}{W_L}$ is injective. Putting this all together, we see that $W_{G,L}$ is naturally identified with a subset of $\quot PGP$. 
	
	Recall that the functor $\ST = \RES^G_{P,L}\IND^G_{P,L}$ has a Mackey filtration with associated graded components $\ST^w$ indexed by $w\in \quot PGP = \quot{W_L}{W}{W_L}$. 
	\begin{proposition}\label{proposition cuspidal mackey}
		If $\fN \in \bD(\ul)$ is cuspidal, then $\ST^w(\fN) \cong 0$ unless $w\in W_{G,L}$. Moreover, if $w\in W_{G,L}$, then 
		\[
		\ST^w(\fN) \cong w_\ast(\fN).
		\]
	\end{proposition}
	\begin{proof} Choose a lift $\dot{w} \in G$ of the double coset $w\in \quot PGP$ such that $\dot{w} \in N_G(T)$. Then, by the Mackey formula, the object
		\[
		\ST^w(\fN) \simeq \IND _{L\cap \ls \dw P, L \cap \ls \dw L} ^{M} \RES^{\ls \dw L} _{P\cap \ls \dw L, L \cap \ls \dw L}
		{\dw _\ast}(\fN)
		\]
		is zero unless $\ls{\dot{w}}{L} = L$, that is $\dot{w} \in N_G(L)$. For the last statement, note that if $w\in W_{G,L}$, then $L \cap \ls \dw L = L$ and the Mackey formula collapses to just ${\dw_\ast}(\fN)$ as claimed. 
	\end{proof}
	
	With this observation in hand, the proof of Theorem \ref{maintheorem monad} is concluded. 
	
	\begin{remark}
	In fact, one can show that the Mackey filtration is compatible with the monad structure in a natural way.\footnote{By way of comparison, recall that a filtered ring is not simply a ring which is filtered as a vector space; rather, the filtration must respect the ring structure in a natural way.}
	We omit this compatibility check as it is not needed for the results of this paper.  
	\end{remark}
	 	
%

	\subsection{The Steinberg functor as an integral transform}\label{sec: steinberg as integral}
	For simplicity, let us first consider the case of the Springer block, i.e. where the cuspidal datum is given by $(T,\C)$ for a maximal torus $T$ of $G$. Fix a Borel subgroup $B$, and recall the Steinberg stack $\St{}{} = \St{B}{B}$ from Section \ref{sec:steinberg-stacks-and-functors}. There are maps
	\[
	\xymatrix{
	\ut & \ar[l]_\alpha \St{}{} \ar[r]^\beta & \ut
	}
	\]
 Let $f= \alpha \times \beta: \St{}{}\to \ut \times \ut$.
	
	Recall the notion of integral transforms for $D$-modules explained in Section \ref{sec:d-modules-on-stacks}. The following lemma is immediate from Proposition \ref{prop-inttrans}.
	\begin{lemma}\label{lemma-monad-inkernel}
	The Steinberg monad 
	\[
	\ST: \bD(\ut) \to \bD(\ut)
	\]
	is represented by the integral kernel $f_\ast(\omega_{\St{}{}}) \in \bD(\ut \times \ut)$. 
	\end{lemma}	

\begin{remark}
	The monad structure on $\ST$ translates in to an algebra structure on $f_\ast \omega_{\St{}{}}$ with respect to the convolution monoidal product in $\bD(\ut \times \ut)$. This corresponds to the structure of fiberwise convolution for the Steinberg stack $\St{}{}$. Over the fiber $(0,0) \in \ft\times \ft$, this restricts to the usual convolution on the equivariant homology of the Steinberg variety as considered in \cite{chriss_representation_1997}.
\end{remark}
 
%
 Now suppose $(L,\cE)$ is a general cuspidal datum for $G$. Let $\St{}{} = \St{P}{Q}$ denote the Steinberg stack, as defined in Section \ref{sec:steinberg-stacks-and-functors}. As in \cite{Gunningham2018}, Section 2, we let $\ul^\heartsuit \subseteq \ul$ denote the substack $\fz(\fl)\times \ucN_L$. Consider the Cartesian diagram:
 \[
 \xymatrix{
 	\Stein \ar[d] & \ar[l] \Stein^\heartsuit \ar[d]^f\\
 	\ul \times \ul & \ar[l] \ul^\heartsuit \times \ul^\heartsuit
 }
 \] 
 Let $\ddot\cE \in \bcD(\ucN_L \times \fz(\fl) \times \ucN_L \times \fz(\fl))$ denote the object $\cE^\vee \boxtimes \omega_{\fz(\fl)} \boxtimes \cE \boxtimes \omega_{\fz(\fl)}$. Unwinding the definitions as in Lemma \ref{lemma-monad-inkernel}, we find:
 
 \begin{lemma}
 	The integral kernel representing $\ST$ is given by 
 	\[
 	f_\ast f^!(\ddot\cE) \in \bD\left(\ul^\heartsuit \times \ul^\heartsuit\right)_{(\cE^\vee \boxtimes \cE)} \simeq \bD(\uz(\fl) \times \uz(\fl)).
 	\]
 \end{lemma}
 
 Note that 
 \[
 \bcD(\uz(\fl))\simeq \ufD_{\fz(\fl)} -\dgmod
 \]
 where $\ufD_{\fz(\fl)} = \fD_{\fz(\fl)} \otimes S_{\fz(\fl)^\ast}$. Thus, we may represent the monad $\ST$ as a dg-algebra $\bA= \bA_{(P,L,\cE)}$, equipped with a morphism of dg-algebras $\fD_{\fz(\fl)} \otimes S_{\fz(\fl)^\ast} \to \bA$. 
 
 With the results of this section in hand, Theorem \ref{maintheorem dg algebra} is simply a rephrasing of Theorem \ref{maintheorem monad}.

	\section{Non-splitting of the Mackey filtration}\label{sec:nonsplit}
In this subsection we will show that the Mackey filtration is non-split in general. Throughout this section, let $B$ denote a Borel subgroup of $G$ with maximal torus $T$. We identify $T \simeq B/U$ where $U$ is the unipotent radical of $B$. We have an adjunction:
\[
\xymatrixcolsep{3pc}\xymatrix{
	\RES := \RES_{B,T}^G: \bD(\ug) \ar@/^0.3pc/[r] &  \ar@/^0.3pc/[l] \bD(\ut):  \IND^G_{B,T} =: \IND.
}
\]
This adjunction induces the Steinberg monad 
\[
\ST = \RES \circ \IND: \bD(\ut) \to \bD(\ut).
\]
The Steinberg monad carries the Mackey filtration indexed by the Weyl group $W=N_G(T)/T$. The goal of this section is to prove Theorem \ref{maintheorem nonsplit}.

\begin{remark}
	It seems likely that the Mackey filtration is non-split for any cuspidal datum $(L,\cE)$ with $L\neq G$. We restrict attention to the trivial cuspidal datum for simplicity and ease of notation.  
\end{remark}

\subsection{The rank 1 case}\label{sec:sl2 computation}
For this subsection, we suppose that $G$ is either $SL_2$ or $PGL_2$. We let $B$ the standard Borel subgroup of upper triangular matrices and $T \simeq \C^\times$ the maximal torus of diagonal matrices (or their image in $PGL_2$). The Weyl group $W$ is generated by a single reflection $s$ acting on $\ft=\A^1$ by $t \mapsto -t$ and on $T$ by $a \mapsto a^{-1}$. The Mackey filtration is given by a single distinguished triangle in $\Fun^L(\bD(\ut),\bD(\ut))$:
\begin{equation}\label{equationfun}
\xymatrix{
	\ST_e \ar[r] & \ST \ar[r] & \ST_s \ar[r]^\delta &.
}
\end{equation}
Note that $\ST_e$ is equivalent to the identity functor $e_\ast$ on $\bD(\ut)$ and $\ST_s$ is equivalent to $s_\ast$. The following proposition proves Theorem \ref{maintheorem nonsplit} in the case $G=SL_2$.
\begin{proposition}\label{propositionnonsplit}
	The connection morphism $\delta$ in the distinguished triangle \ref{equationfun} is non-zero.
\end{proposition}

In order to prove Proposition \ref{propositionnonsplit}, recall that the category $\Fun^L(\bD(\ut),\bD(\ut))$ is equivalent to $\bD(\ut \times \ut)$ (see Proposition \ref{prop-inttrans}), and thus we may replace the functors $\ST$, $\ST_e$ and $\ST_s$ by their integral kernels to obtain a distinguished triangle in $\bD(\ut \times \ut)$:
\begin{equation}\label{equationk}
\xymatrix{
	\fK_e \ar[r] & \fK \ar[r] & \fK_s \ar[r]^\delta & 
}
\end{equation}
By Example \ref{example-int-kernel}, $\fK = f_\ast(\omega_{\St{}{}})$. The stratification of the Steinberg stack $\St{}{} = \Stw{}{}e \cup \Stw{}{}s$ gives rise to the distinguished triangle \ref{equationk}. To make the computation more explicit, consider the following variant of the Steinberg variety:
\[
\Stein = \{ (x,g) \in \fg \times G \mid x \in \fb, \ls gx \in \fb \}
\]
The Steinberg stack $\St{}{} = \St{B}{B}$ is the stack quotient $\Stein/(B\times B)$.  The projection maps are given by:
\[
\xymatrix{
	\ft & \ar[l]_a \Stein \ar[r]^b & \ft \\
	x + \fu & \ar@{|->}[l] (x,g) \ar@{|->}[r] & \ls gx + \fu.
}
\]
The variety $\Stein$ is the union of a closed stratum $\Stein^e$, given by the locus where $g\in B$, and an open stratum $\Stein^s$ where $g\notin B$. We write $f = (a\times b): \Stein \to \ft \times \ft$ and $f_e$, $f_s$ for the restrictions to $\Stein^e$, $\Stein^s$ respectively. We have a distinguished triangle
\begin{equation}\label{equationf}
\fF_e := f_{e\ast} \omega_{\Stein^e}  \to\fF :=  f_\ast \omega_\Stein \to \fF_s := f_{s\ast} \omega_{\Stein^s} \xrightarrow{\delta}.
\end{equation}
(Taking the fiber over a point $(t_1,t_2) \in \ft \times \ft$, the distinguished triangle \ref{equationf} computes the long exact sequence in Borel-Moore homology of the fiber $\Stein_{(t_1,t_2)}$ associated to the partition $\Stein_{(t_1,t_2)} =  \Stein_{(t_1,t_2)}^e \cup \Stein_{(t_1,t_2)}^s$.)

Up to a cohomological shift,\footnote{Strictly speaking, the variety $\Stein$ is not obtained from the stack $\St{}{}$ by base change from $\ut \times \ut$ to $\ft \times \ft$, but it differs from the base change only by unipotent gerbes whose only effect on the category of $D$-modules is a cohomological shift.} the distinguished triangle \ref{equationf} is obtained from \ref{equationk} by forgetting the equivariant structure. Thus we are reduced to proving that the connecting morphism $\delta$ in the distinguished triangle \ref{equationf} is non-zero.

Let $\Delta: \ft \hookrightarrow \ft \times \ft$ denote the diagonal, and $\nabla: \ft \hookrightarrow \ft \times \ft$ the antidiagonal (i.e. the graph of $s: \ft \to \ft$). Note that $\fF_e$ is supported $\Delta$ whereas $\fF_s$ is supported in $\nabla$, and that the images of $\Delta$ and $\nabla$ intersect at $(0,0)$. 

Let $p: \ft \to pt$ and consider the distinguished triangle of complexes of vector spaces:
\[
p_\ast \nabla^! \fF_e \to p_\ast \nabla^! \fF \to p_\ast \nabla^! \fF_s \to
\]
This triangle gives rise to the long exact sequence in Borel-Moore homology associated to $X= \nabla^{-1}(\Stein)$ with its partition into a closed subset $ X_e = \nabla^{-1}(\Stein_e)$ and open complement $X_s = \nabla^{-1}(\Stein_s)$. Explicitly, we have:
\[
X = \{ (x,g) \in \fg \times G \mid x\in \fb, \ls gx  \in \fb, x + \fu = -\ls gx + \fu \}.
\]
Note that the fiber $h^{-1}(g)$ is given by $\{g\} \times \fu$ if $G\in B$, or $\{g\}\times \fb\cap \ls{g^{-1}}\fb$ if $g\notin B$. In particular, all the fibers are isomorphic to $\A^1$. In fact, we have:
\begin{lemma}
	The map $X\to G$ is a line bundle.
\end{lemma} 
\begin{proof}
	We first assume that $G=SL_2$. Observe that the morphism
	\begin{align*}
	\A^1 \times G & \to \fg \times G \\
	\left(
	t,
	\begin{bmatrix}
	a             & b                \\
	c             & d                
	\end{bmatrix}
	\right)
	& \mapsto          
	\left(
	t\cdot \begin{bmatrix}
	c            & 2d              \\
	0             & -c              
	\end{bmatrix},
	\begin{bmatrix}
	a             & b                \\
	c             & d                
	\end{bmatrix}
	\right) 
	\end{align*}
	defines an isomorphism of $\A^1\times G$ onto $X\subseteq \fg\times G$. In other words, $X\to SL_2$ is the trivial line bundle. For the $PGL_2$ case, note that the above map is equivariant for the action of the center $Z(SL_2)$, where $Z(SL_2)$ acts on $\A^1$ by the unique non-trivial character. Thus we get the line bundle $\A^1 \times^{Z(SL_2)} SL_2$ over $PGL_2$ as required. 
\end{proof}

\begin{proof}[Proof of Proposition \ref{propositionnonsplit}]
Let us first consider the case $G=SL_2$. We are reduced to showing that the connecting morphism in the long exact sequence
\[
H_\ast^{BM}(X_e) \to H_\ast^{BM}(X) \to H_\ast(X_s) \to
\]
is non-zero. By the above calculations, we see that $X$ is homeomorphic to $S^3 \times \R^5$, $X_e$ is homeomorphic to $S^1 \times \R^5$, and $X_s$ is homeomorphic to $S^1 \times \R^7$. Thus the long exact sequence in Borel-Moore homology takes the following form:
\[
\xymatrix{
	0 \ar[r] & \C[8] \ar[r] & \C[8] \ar[lld] \\
	0 \ar[r] & 0 \ar[r] & \C[7] \ar[lld] \\
	\C[6]\ar[r] & 0 \ar[r] & 0 \ar[lld]\\
	\C[5] \ar[r] & \C[5] \ar[r] & 0
}
\]
We deduce that the connecting morphism must be non-zero as required. This concludes the argument for $G=SL_2$. Finally, the Borel-Moore homology groups are unchanged in the case $G=PGL_2$, so the same argument applies. 
\end{proof}

\subsection{Proof of Theorem \ref{maintheorem nonsplit}}
The goal of this subsection is to prove the following result, which implies Theorem \ref{maintheorem nonsplit}.

\begin{proposition}\label{proposition nonsplit general}
	The unit morphism $u:\id_{\bD(\ut)} \to \ST$ for the Steinberg monad does not split. That is, there does not exist a natural transformation $s:\ST \to \id_{\bD(\ut)}$ such that $su$ is the identity natural transformation. 
\end{proposition}
\begin{proof}
We first note that Proposition \ref{proposition nonsplit general} holds in the case $G=SL_2$ or $G=PGL_2$ by Proposition \ref{propositionnonsplit}. It follows immediately that the result holds for $G$ of the form $G'\times Z$ where $G'$ is semisimple of rank 1 and $Z$ is a torus.  

More generally, we claim that the result holds for any reductive group $G$ of semisimple rank $1$. Indeed, any such $G$ is of the form $\widetilde{G}/\Gamma$ where $\widetilde{G} = [G,G]\times Z^\circ(G)$ is a product of the semisimple group $G'=[G,G]$ with the torus $Z^\circ(G)$, and $\Gamma = Z([G,G])\cap Z^\circ(G)$ is a finite central subgroup. In that case, map $\widetilde{\ug} =\fg/\widetilde{G} \to \ug = \fg/G$ is a $\Gamma$-gerbe and thus the pullback functor $\bD(\ug) \to \bD(\widetilde{\ug})$ is an equivalence of categories. Similarly for the corresponding maps $\bD(\ub) \to \bD(\fb/\widetilde{B})$ and $\bD(\ut) \to \bD(\ft/\widetilde{T})$ where $\widetilde{B}$ and $\widetilde{T}$ are the corresponding Borel subgroup and maximal torus in $\widetilde{G}$. Thus, the Steinberg monad for $\widetilde{G}$ is naturally identified with the Steinberg monad for $\widetilde{G}$, which we have already seen does not split. 

Finally, suppose that we are in the setting of Theorem \ref{maintheorem nonsplit}. We may choose a parabolic subgroup $P\supseteq B$ of $G$ with Levi factor $M\supseteq T$ of semisimple rank 1. Then $B_M = B\cap M$ is a Borel subgroup in $M$, and we have maps
\[
\xymatrix{\id_{\bD(\ut)} \ar[r]_-{u_M} \ar@/^1pc/[rr]^-u&\ST^M := \RES^M_{B_M,T} \IND^M_{B_M,T} \ar[r]_-{v} & \ST = \RES^G_{B,T} \IND^G_{B,T}}
\] 
Here, the map $v$ is induced by unit morphism $\id_{\bD(\um)} \to \ST^G_M = \RES^G_{P,M} \IND^G_{P,M}$ in light of the equivalence $\RES^G_{B,T} \simeq \RES^G_{P,M} \RES^M_{B_M,T}$ (see \cite[Proposition 1.7]{Gunningham2018}). A splitting of the map $u$ induces a splitting of the map $u_M$ which we have seen does not exist. Thus the map $u$ does not split, as required. 
\end{proof}

\subsection{Proof of Theorem \ref{mainthm opposite}}
We keep the notation introduced at the beginning of this section. In addition, let $\overline{B}$ be the opposite Borel subgroup containing $T$ and let $\overline{\RES} = \RES^G_{\overline{B},T}$. The goal of this subsection is to prove Theorem \ref{mainthm opposite}, which we restate below.

\begin{theorem}
	There is no natural isomorphism of functors between $\RES^G_{B,T}$ and $\RES^G_{\overline{B},T}$. 
\end{theorem}

We will prove this result by contradiction. To that end, we henceforth suppose that there is a natural isomorphism of functors:
\begin{equation}\label{isomorphism of functors}
\RES^G_{B,T} \cong \RES^G_{\overline{B},T}
\end{equation}

We have the following second adjunction theorem due to Drinfeld--Gaitsgory (following Braden). 
\begin{theorem}[{\cite[Theorem 3.4.3]{drinfeld_theorem_2014}, \cite{braden_hyperbolic_2003}}, {\cite[Theorem 3.7]{Gunningham2018}}] \label{theorem second adjunction}
	The functor $\RES^G_{\overline{B},T}$ is left adjoint to $\IND^G_{B,T}$.
\end{theorem}

Under our assumption (\ref{isomorphism of functors}) it follows that $\RES$ is both right and left adjoint to $\IND$. Let 
\[
u_1: \id_{\bD(\ut)} \to \ST = \RES \IND
\]
be the unit of the first adjunction, and 
\[
c_2: \ST \to \id_{\bD(\ut)}
\]
the counit of the second adjunction. The Mackey filtration  the inclusion of $\{e\} \hookrightarrow W$
Note that the exact triangle formed by the unit $u_1$ coincides with the part of the Mackey filtration induced by the inclusion of the downward-closed subposet $\{e\} \hookrightarrow W$:
\[
\ST^{e} = \id_{\bD(\ut)} \xrightarrow{u_1} \ST \to \ST^{>e}.
\]
Proposition \ref{proposition nonsplit general} thus implies the following lemma.
\begin{lemma}\label{lemma nonsplit}
The unit $u_1$ is not split; that is, there does not exist a natural transformation $s:\ST \to \id_{\bD(\ut)}$ such that $s u_1 = \id_{\id_{\bD(\ut)}}$.
\end{lemma}

\begin{lemma}\label{lemma cu=0}
	We have $c_2u_1=0$.
\end{lemma}
\begin{proof}
	By the theory of integral transforms, and noting that the identity functor corresponds to $\Delta_\ast(\omega_{\ut})$ where $\Delta:\ut \to \ut \times \ut$ is the diagonal morphism, we have that:
	\[
	\End(\id_{\bD(\ut)}) \cong \End_{\bD(\ut \times \ut)}(\Delta_\ast \omega_{\ut}) \cong H_{dR}^0(\ut) \cong \C.
	\]
	In particular, the composition $c_2u_1 = \lambda \id_{\id_{\bD(\ut)}}$ for some constant $\lambda \in \C$.
If $\lambda \neq 0$, we can rescale $c_2$ to split $u_1$, contradicting Lemma \ref{lemma nonsplit}. 
\end{proof} 

We will now derive a contradiction to Lemma \ref{lemma cu=0}. Let us recall some notions from \cite[Section 3.1]{Gunningham2018} regarding parabolic restriction over the regular locus. Let $j: \ut^{\reg} \hookrightarrow \ut$ denote the inclusion of the regular locus and $k:\ug^{rs} \hookrightarrow \ug$, the inclusion of the regular semisimple locus. Let $\ub^\reg = \ub \times_{\ut} \ut^\reg$.

Let $d: \ut \to \ug$ denote the natural map induced by the inclusion $T\hookrightarrow G$, and $d^\reg: \ut^{\reg} \to \ug^{\rs}$ the restriction to the regular locus. Then, as explained in \cite[Section 3.1]{Gunningham2018}, we have a commutative diagram of stacks,
\[
\xymatrixcolsep{3pc}
\xymatrix{
		\ug &\ar[l]_-r  \ub  \ar[r]^-s & \ut\\
	\ug^\reg  \ar[u]^k &\ar[l]_-{r^\reg}  \ub^\reg  \ar[u] \ar[r]^-{s^\reg}_-\sim& \ar@/^1pc/[ll]^-{d^\reg} \ut^{\rs} \ar[u]_j,
	}
\]
where the vertical maps are open embeddings, $s^\reg$ is an isomorphism, both squares are cartesian, and $d^\reg$ is a $W$-Galois covering. Note that $j^!j_\ast \cong \id_{\bD(\ut^\reg)}$ (respectively, $k^!k_\ast \cong \id_{\bD(\ug^\rs)}$) as $j$ (respectively $k$) is an open embedding.
Over the regular locus we have the functors
\begin{align*}
	\RES^\reg &:= j^! \RES k_\ast \cong (s^\reg)_\ast(r^\reg)^! \cong (d^{\reg})^!\\
		\IND^{\reg} &:= k^! \IND j_\ast \cong (r^\reg)_\ast (s^\reg)^! \cong (d^\reg)_\ast.  
\end{align*}

By base-change, we have $\RES k_\ast \cong j_\ast \RES^\reg$ and $\IND j_\ast \cong k_\ast \IND^\reg$. It follows that the bi-adjoint functors $\IND$ and $\RES$ restrict to bi-adjoint functors between the full subcategories $\bD(\ug^\rs)$ and $\bD(\ut^\reg)$. Moreover, the natural transformations $c_2^\reg := j^! c_2 j_\ast$ and $u_1^\reg := j^! u_1 j_\ast$ are counit and unit morphisms for these adjunctions (here, the juxtaposition of a functor with a natural transformation denotes whiskering). 

On the other hand, as $d^\reg$ is a principal bundle for a finite group (so in particular, smooth and proper), $\RES^\reg = (d^\reg)^!$ is both right and left adjoint to $\IND^\reg = (d^\reg)_\ast$, and (crucially), the counit $c'_2:(d^\reg)^!(d^\reg)_\ast \to \id_{\bD(\ut^\reg)}$ of the second adjunction splits the unit $u'_1:  \id_{\bD(\ut^\reg)} \to (d^\reg)^!(d^\reg)_\ast$ of the first adjunction. Indeed, for $\fM \in \bD(\ut^\reg)$,
\[
(d^\reg)^!(d^\reg)_\ast(\fM) = \bigoplus_{w\in W} w_\ast(\fM),
\]
and $u_1'$ (respectively $c_2'$) is given by inclusion of (respectively, projection onto) the summand corresponding to the identity $e\in W$. 

Thus, we are in the following situation: 
\begin{itemize}
	\item we have a pair of functors $\RES^\reg$, $\IND^\reg$; 
	\item $c_2^\reg,c_2'$ are two counit morphisms exhibiting $\RES^\reg$ as left adjoint to $\IND^\reg$; and 
	\item $u_1,u_1'$ are two unit morphisms exhibiting $\RES^\reg$ as right adjoint to $\IND^\reg$. 
\end{itemize}

Finally, we recall the following form of uniqueness for adjunctions.
\begin{lemma}
Let $F:\cC \to \cD$ and $G:\cD \to \cC$ be functors. If $c, c':FG \to \id_{\cD}$ are two counit morphisms witnessing $F$ as left adjoint to $G$, then there exists $\alpha\in \Aut(G)$ such that $c'$ is given by the composite
\[
GF \xrightarrow{\alpha \id_{F}} GF \xrightarrow{c} \id_{\cD}.
\]
Similarly, if $u,u':\id_\cD \to FG$ are two unit morphisms witnessing $F$ as right adjoint to $G$, then there exists $\beta \in \Aut(G)$ such that $u'$ is given by the composite
 \[
 GF \xrightarrow{\beta \id_{F}} GF \xrightarrow{u} \id_{\cD}.
 \]
\end{lemma}

In our case, we have that 
\[
\Aut(\RES^\reg) \cong \Aut_{\bD(\ug^{\rs} \times \ut^\reg)}((\ub^\reg \to \ug_\rs \times \ut^\reg)_\ast (\omega_{\ub^\reg})) \cong H^0(\ub^\reg)^\times \cong \C^\times.
\]
It follows that any two counits (respectively, units) for an adjunction between $\RES^\reg$ and $\IND^\reg$ differ by a non-zero scalar. As $c_2'u_1' = \id_{\bD(\ut^\reg)}$, it follows that $c_2^\reg u_1^\reg \neq 0$. This contradicts Lemma \ref{lemma cu=0}. It follows that there cannot exist a natural isomorphism of functors as in (\ref{isomorphism of functors}). This concludes the proof of Theorem \ref{mainthm opposite}.

\bibliography{papers_jabref}
\bibliographystyle{amsalpha}

\end{document}